\documentclass[a4paper,11pt,american]{amsart}
\usepackage{mathrsfs}
\usepackage{dsfont}
\usepackage{mathtools}
\usepackage{amssymb}
\usepackage[foot]{amsaddr}
\usepackage{cite}
\usepackage{enumitem}
\usepackage{hyperref}
\usepackage{todonotes}
\usepackage{subfigure}
\usepackage{geometry}
\usepackage{colortbl}
\def\R{\mathbb R}
\def\N{\mathbb N}

\def\y{\mathbf y}
\def\ep{\varepsilon}
\def\ds{\displaystyle}
\def\tb{\overline{\theta}}
\def\tmin{\theta_{\textup{min}}}
\def\tmax{\theta_{\textup{max}}}
\def\ub{\overline{u}}
\def\rhob{\overline{\rho}}

\def\lp{\left(}
\def\rp{\right)}

\definecolor{Lionel}{rgb}{0.13, 0.68, 0.8}

\newtheorem{thm}{\textbf{Theorem}}[section]
\newtheorem{lem}[thm]{\textbf{Lemma}}
\newtheorem{prop}[thm]{\textbf{Proposition}}

\theoremstyle{remark}
\newtheorem{rem}[thm]{Remark}

\numberwithin{equation}{section}

\title[Allee threshold as phenotypic trait: persistence vs. extinction]{When the Allee threshold is an evolutionary trait: persistence vs. extinction}
\author{Matthieu Alfaro}
\address[M. A.]{Universit\'e de Rouen Normandie, CNRS, Laboratoire de Math\'ematiques Rapha\"el Salem, Saint-Etienne-du-Rouvray, France \& BioSP, INRAE, 84914, Avignon, France.}
\email{matthieu.alfaro@univ-rouen.fr}
\author{L\'{e}o Girardin}
\address[L. G.]{Univ Lyon, CNRS, Universit\'e Claude Bernard Lyon 1, UMR 5208, Institut Camille Jordan, F-69622 Villeurbanne, France}
\email{leo.girardin@math.cnrs.fr}
\thanks{This work was supported by a public grant as part of Investissement d'Avenir projects, refe\-rences ANR-11-LABX-0056-LMH, LabEx LMH, and ANR-11-IDEX-0001-02, A*MIDEX. This work has been carried out in the framework of
the NONLOCAL (ANR-14-CE25-0013) and RESISTE (ANR-18-CE45-0019) projects funded by the French National Research Agency
(ANR). This work was also supported by the 
ANR \textsc{i-site muse}, project \textsc{michel} 170544IA (n$^{\circ}$ ANR \textsc{idex}-0006).}
\author{Fran\c cois Hamel}
\address[F. H.]{Aix Marseille Univ, CNRS, Centrale Marseille, I2M, Marseille, France.}
\email{francois.hamel@univ-amu.fr}
\author{Lionel Roques}
\address[L. R.]{BioSP, INRAE, 84914, Avignon, France.}
\email{lionel.roques@inrae.fr}

\begin{document}

\begin{abstract}
We consider a nonlocal parabolic equation describing the dynamics of a popu\-lation structured by a spatial position and a phenotypic trait, submitted to dispersion, mutations and growth. The growth term may be of the Fisher-KPP type but may also be subject to an Allee effect which can be weak (non-KPP monostable nonlinearity, possibly degenerate) or strong (bistable nonlinearity). The type of growth depends on the value of a variable $\theta:$ the Allee threshold, which is considered here as an evolutionary trait. After proving the well-posedness of the Cauchy problem, we study the long time behavior of the solutions. Due to the richness of the model and the interplay between the various phenomena and the nonlocality of the growth term, the outcomes (extinction vs. persistence) are various and in sharp contrast with earlier results of the existing literature on local reaction-diffusion equations.   
\end{abstract}

\keywords{Allee effect, evolutionary rescue, reaction-diffusion, structured population}
\subjclass[2010]{35K57, 35R09, 92D15, 92D25}
\maketitle

\tableofcontents

\section{Introduction}\label{sec1}

\subsection{Position of the problem}
We consider a population density $u=u(t,x,\theta)$, which depends on time $t\ge0$ and is structured by a spatial position $x\in \R$ and a phenotypic trait
$$\theta\in \Theta:=(\tmin,\tmax),$$
where
$$-\infty<\tmin<\tmax<1,$$
and whose evolution is governed by the nonlocal problem
\begin{equation}
\label{eq}
u_t=du_{xx}+\alpha u_{\theta \theta}+u(\rho-\theta)(1-\rho) \quad\text{for all } t>0,\ x\in \R,\ \theta\in\Theta.
\end{equation}
In~\eqref{eq}, $d>0$ is the spatial diffusion coefficient, $\alpha> 0$ is a coefficient which measures the impact of mutations on the trait (mutation rate $\times$ mutation effects, see Appendix~A in \cite{HamLav20}). The nonlocal term $\rho=\rho(t,x)$ corresponds to the total population density at spatial position $x$; it depends on the solution $u=u(t,x,\theta)$ itself and is given by
\begin{equation*}
\label{nonlocal}
\rho(t,x):=\int_{\Theta} u(t,x,\theta)\,\textup{d}\theta\quad\text{for all }t>0,\ x\in\R.
\end{equation*}
To ensure that mutations have no effect on the population size, the equation~\eqref{eq} is supplemented with no-flux boundary conditions on the boundary of the interval $\Theta$ of phenotypic traits, namely
\begin{equation}
\label{neumann}
u_\theta(t,x,\tmin)=u_\theta(t,x,\tmax)=0\quad\text{for all }t>0,\ x\in \R.
\end{equation}

Our first main concern is to perform a detailed analysis of the Cauchy problem obtained by supplementing to \eqref{eq}--\eqref{neumann} an initial condition
\begin{equation}
\label{initial}
u(0,\cdot,\cdot)=u_0\in\mathcal{C}_c(\R\times\overline{\Theta}, [0,+\infty))\text{ with } M:=\sup_{x\in\R}\int_{\Theta}u_0(x,\theta)\,\textup{d}\theta>0.
\end{equation}
In~\eqref{initial}, $\mathcal{C}_c$ denotes the space of continuous functions with compact support (hence, $M$ is necessarily finite). The initial conditions are here assumed moreover to be nonnegative. Secondly, we will investigate the long time dynamics (persistence vs. extinction) of the population density $u$ and its mass $\rho$, according to the value of the parameters~$\alpha$,~$\tmin$ and~$\tmax$, and according to the initial condition $u_0$.

\subsection{Biological context}

A biological invasion is generally considered as a three-stage process~\cite{BlaPys11}, which begins with the introduction of some individuals into a new environment, and is followed by the establishment and spreading of the population. This scenario corresponds to a successful invasion. However, individuals that arrive from a source popu\-lation into a new environment do not necessarily establish a new population~\cite{BlaPys11}. Either the newly-introduced individuals are well-adapted to the new environment and can readily establish or the introduced population declines due to a negative growth rate. In the latter case, evolutionary adaptation can lead to the establishment of such initially declining populations. This process is known as {\it evolutionary rescue} \cite{ColAle17}.

The success of an invasion depends on several factors that include characteristics of the species, of the introduction event (initial population size and spatial distribution \cite{GarRoqHam12} and genotype diversity), and of the new environment (e.g., climate matching or presence of hosts for biological invasions). Among these factors, several studies have shown that the presence of an Allee effect, a decreased individual fitness at low population density \cite{All38},  plays an important role, as introductions typically involve small populations~\cite{Dra04,LeuDra04,YamLie09}. The Allee effect may result from several simultaneous processes \cite{BerAng07} that arise at low densities, such as diminished chances of finding mates or inbreeding  \cite{CouBer08}.

Many spatio-temporal modelling approaches that focused on invasion success in the presence of an Allee effect adopted a purely demographic viewpoint \cite{LewKar93,GarRoqHam12,DruDra07},  thereby neglecting the effects of genetic adaptation. These studies were based on PDE reaction-diffusion models with growth functions of the form $f(\rho)= \rho(\rho-\theta_0)(1-\rho)$, with $\theta_0<1$ a given constant representing the Allee threshold \cite{SteSut99}. We recall that this function may account for the presence of a strong Allee effect if $\theta_0>0$, a weak Allee effect if $\theta_0 \in (-1,0]$ and can also take a KPP form without Allee effect if $\theta_0\le -1$ \cite{Tur98} (see also Table~\ref{table:summary_standard}).

Here, we take a different approach to analyze the success of an invasion in the presence of genetic adaptation when the trait under selection precisely corresponds to the strength~$\theta$ of the Allee effect. In the source population, where the introduced individuals come from, Allee effects may have been promoted by evolution, due to high population sizes \cite{BerKra17}. In such dense populations, where mate finding is easy, individual may indeed acquire traits that impair fitness at lower densities. Conversely, the selection pressure at low density can promote traits that reduce the strength of the Allee effect. More generally, this type of density-dependent selection \cite{Rou71,Asm83} can occur when the fitness associated to a trait value depends on the population density. Empirical examples include dispersal dimorphism in several insect species, where individuals with higher dispersal potential (and thus better mate-finding ability) mostly appear at low population densities \cite{ZerDen97}. The invasion of cane toads in Australia is another well-documented example of this dimorphism \cite{KelShi20}. Yet another example of density-dependent selection is the evolution of virulence in spreading epidemics \cite{Griette_Raoul_}. Recently, such biological problems and the underlying trade-offs have attracted a great deal of interest from mathematical modelers \cite{Ben-Cal-Meu-Voi-12,Elliott_Cornel,Griette_Raoul_,Keenan_Cornell_2020}. 

As the leading edge of an invasion is by definition a region where the population density is low, one may expect that important selection pressure on traits that regulate the Allee effect occurs there. The strength of the Allee effect is known to have important effects on the persistence/extinction and the spread of invasive organisms, we therefore expect that the evolution of these traits have important effect on invasion dynamics. Yet, this question has not been addressed theoretically until the recent work \cite{EP-20}, which is based on individual-based simulations of a model with a trait that governs resistance to the Allee effect. In \cite{EP-20}, the authors focused on the transition during the course of an invasion from \textit{pushed waves}, namely steep waves where the whole bulk of the wave pushes the invasion forward and the spreading speed is nonlinearly determined, to \textit{pulled waves}, namely flatter waves that are driven only by the exponential tail ahead of the front and whose spreading speed is linearly determined. Here, we rather focus on the conditions that lead to a successful invasion or not.

\subsection{Mathematical context}

In the last decade, mathematical population models structured by both a spatial and a trait variable and evolving in an unbounded spatial domain have received much attention. As far as Fisher-KPP growth terms (the per capita growth rate is decreasing with respect to the population density) are concerned, let us mention the works~\cite{Alfaro_Coville_Raoul,Berestycki_Jin_Silvestre,Bouin_Mirrahimi,Griette_2017} dealing with constant motility, \cite{Arnold_Desvill,Berestycki_Mouhot_Raoul,Bouin_Calvez_2014,Bou-Cal-Meu-Mir-Per-Rao-12,Bouin_Chan_Hen,Bouin_Henderso,BHR-17,Cal_Hen_Mir_Tur_Dum,Girardin_Griette_2020,Turanova_2015} where the motility is trait-dependent (cane toad equation) and \cite{Alfaro_Griette,Girardin_2016_2,Girardin_2016_2_add,Girardin_2017,Girardin_2018,Girardin_Griette_2020,Griette_Raoul,Morris_Borger_Crooks} where the trait structure is discrete. As far as bistable growth terms are concerned, let us mention~\cite{Bouin_Henderson} where a local bistable-type equation is concerned, and~\cite{AD-18} considering a nonlocal model proposed in~\cite{KW-10} for evolutionary rescue.

The model \eqref{eq} is not only space-trait structured but also includes a trait-dependent Allee effect in the growth term (possibly threatening small populations). Moreover, due to the nonlocality of the growth term and its non-monotonicity with respect to $u$ (remember that $\rho$ depends on $u$), the comparison principle does not hold in general for~\eqref{eq}--\eqref{initial}, that is, even if two initial conditions $u_0$ and $v_0$ are ordered, the solutions emanating from them may not be ordered at positive times.

As mentioned above, the main purpose of this work is to determine conditions that imply persistence or extinction of a population whose density is governed by \eqref{eq}--\eqref{initial}. Since the Allee threshold, or the strength of the Allee effect, is regarded as an evolutionary trait subject to mutations and selection, the model under consideration may share some similarities with various classical local reaction-diffusion equations such as Fisher-KPP, degenerate monostable, or bistable equations. As a result, the model can reveal many phenomena which are common in the study of local reaction-diffusion equations. Let us first mention the so-called \textit{hair trigger effect}~\cite{AW-78}, meaning that persistence occurs whatever the size of the initial density. Notice that the hair trigger effect is related to the seminal  blow-up result of Fujita \cite{Fuj-66}. On the other hand, some \textit{threshold} phenomena \cite{ADF-19, Du-Mat-10, P-11, Zla06} may occur, meaning that \lq\lq small'' populations typically go extinct whereas \lq\lq large'' populations typically persist. These classical results are summarized in Table~\ref{table:summary_standard} for the standard model $\rho_t=d \rho_{xx}+\rho(\rho-\theta_0)(1-\rho)$, where~$\theta_0$ is a fixed parameter that controls the occurrence of an Allee effect, see above. For the model \eqref{eq}, we expect a much more complicated behavior. We distinguish between three possible scenarios: hair trigger effect, possible persistence or extinction depending on the initial condition, and systematic extinction whatever the (compactly supported) initial condition. As we will see, the range~$(\tmin,\tmax)$ over which the trait may vary plays a critical role on the fate of the population.

\renewcommand{\arraystretch}{1.5}
{\Small\begin{table}
\noindent\hspace{-.5cm}\begin{tabular}{c|c|c|c|c|c|}
    \cline{2-6}
    & $\theta_0 \le -1$ & $\theta_0 \in (-1, -1/2]$ & $\theta_0\in (-1/2,0]$  & $\theta_0 \in (0,1/2)$ & $\theta_0\in[1/2,1)$ \\
    \hline
    \multicolumn{1}{|c|}{Outcome} & P. & P. & P.  &  E. or P. & E. \\
    \hline
    \multicolumn{1}{|c|}{Strength of the A. E.} & No A. E. & Weak A. E. & 
Weak A. E. & Strong A. E. & Strong A. E. \\
    \hline
    \multicolumn{1}{|c|}{Nature of the front} & Pulled & Pulled  & Pushed 
  &  Pushed  & Pushed \\
    \hline 
    \multicolumn{1}{|c|}{Spreading speed} & $2\sqrt{-\theta_0 d}$ & $2\sqrt{-\theta_0 d}$ & $\sqrt{2d} (1/2-\theta_0)$ & $\sqrt{2d} (1/2-\theta_0)$ & 0 \\
    \hline
\end{tabular}
\vskip 0.3cm
\caption{Standard persistence and spreading results for the equation $\rho_t=d \rho_{xx}+\rho(\rho-\theta_0)(1-\rho)$ with compactly supported initial condition $\rho_0\ge,\not\equiv0$ (here, $\theta_0<1$ is a fixed parameter). P.: systematic persistence independently of $\rho_0$ (hair trigger effect); E.~or~P.: outcome depending on $\rho_0$; E.: systematic extinction independently of $\rho_0$; A. E.: Allee effect. The front refers to the unique front or to the front with minimal speed. Its pulled/pushed nature is understood in the sense of \cite{Gar-Gil-Ham-Roq-12}.}
\label{table:summary_standard}
\end{table}
}

If survival occurs, one may like to analyze the propagation phenomena, in particular to determine the \textit{spreading speed} which is related to the nature  of the traveling front \cite{Gar-Gil-Ham-Roq-12}. We believe that the model may exhibit fronts that, in some sense, may switch from \textit{pushed} to \textit{pulled}, as observed through individual-based models \cite{EP-20}. We plan to address such an issue in a future work.

\subsection{Summary of the main results} 

We here briefly comment our main results, which will be clarified throughout the paper.

We start by proving important {\it a priori} estimates and the well-posedness of the Cauchy problem \eqref{eq}--\eqref{initial}. In particular the solutions $u$ of~\eqref{eq}--\eqref{initial} are understood in the classical sense, namely of class $\mathcal{C}^{1;2}_{t;(x,\theta)}((0,+\infty)\times\R\times\overline{\Theta})\cap\mathcal{C}([0,+\infty)\times\R\times\overline{\Theta})$ (and therefore~\eqref{eq} is satisfied in $(0,+\infty)\times\R\times\overline{\Theta}$). The mass $\rho$ over the trait space will then be of class~$\mathcal{C}^{1;2}_{t;x}((0,+\infty)\times\R)\cap\mathcal{C}([0,+\infty)\times\R)$.

Then our main goal is to figure out the long time behavior of the solutions. We first define the Neumann principal eigenpair $(\lambda_\alpha,\varphi_\alpha)$ of the linearized operator around the trivial steady state $0$ corresponding to perturbations that vary only in the $\theta$ variable, namely
\begin{equation*}
\begin{cases}
-\alpha \varphi_\alpha ''+\theta\varphi_\alpha =\lambda_\alpha \varphi_\alpha & \text{ in } \overline{\Theta},\vspace{3pt}\\
\varphi_\alpha'(\tmin)=\varphi_\alpha'(\tmax)=0,\vspace{3pt}\\
\varphi_\alpha>0& \text{ in }\overline{\Theta}.\vspace{3pt}
\end{cases}
\end{equation*}
We prove that the outcome of the population, extinction or persistence, depends on a subtle combination of the sign of $\lambda_\alpha$, the range of admissible phenotypic traits $\Theta=(\tmin,\tmax)$, and the initial density $u_0$, as summarized in Table~\ref{table:summary}. By extinction~(E.), we mean that $\|u(t,\cdot,\cdot)\|_{L^\infty(\R\times\Theta)}\to0$ as $t\to+\infty$. By persistence (P.), we mean the opposite, that is, $\limsup_{t\to+\infty}\|u(t,\cdot,\cdot)\|_{L^\infty(\R\times\Theta)}>0$. We will also see that these definitions have equivalent formulations for the mass $\rho$. 

\renewcommand{\arraystretch}{1.5}
{\Small\begin{table}
\noindent\hspace{-1.5cm}\begin{tabular}{c|c|c|c|c|c|c|}
    \cline{2-7}
     & $\tmin <0$ and & $\tmin< 0$ and & $0\!\leq\!\tmin\!<\!1/2$ and & $0\!\le\!\tmin\!<\!1/2$ and & $0\!\le\!\tmin\!<\!1/2$ and & $\tmin\geq 1/2$ \\
     & $\tmin\!+\!\tmax\!\le\!0$ & $\tmin\!+\!\tmax\!>\!0$ & $\tmin+\tmax<1$ & $\tmin+\tmax=1$ & $\tmin+\tmax>1$ & \\
    \hline
    \multicolumn{1}{|c|}{$\lambda_\alpha<0$} & P.  & P. & N/A & N/A & N/A & N/A \\
    \hline
    \multicolumn{1}{|c|}{$\lambda_\alpha=0$} & N/A & P. & N/A & N/A & N/A & N/A \\
    \hline
    \multicolumn{1}{|c|}{$\lambda_\alpha>0$} & N/A & E. or P. &  E. or P. &   \cellcolor{yellow} E. or P.  & \begin{tabular}{@{}l@{}}
                   $\exists\,\alpha^\star>0$,\\
                   \cellcolor{yellow}E. or P. if $\alpha\le\alpha^\star$\\
                   E. if $\alpha>\alpha^\star$\\
                 \end{tabular}   & E. \\
    \hline
\end{tabular}
\vskip 0.3cm
\caption{Summary of the main results. E.: systematic extinction independently of $u_0$ satisfying~\eqref{initial}; P.: systematic persistence independently of $u_0$ satisfying~\eqref{initial}; E.~or~P.: outcome depending on $u_0$; colored cells: the possibility of extinction is proved, the possibility of persistence is conjectured but not proved; N/A: not applicable, \textit{i.e.} impossible case.
}
\label{table:summary}
\end{table}
}

As will be seen in Section~\ref{sec:ev}, the map $\alpha\mapsto\lambda_\alpha$ is an increasing concave bijection from~$(0,+\infty)$ onto the open interval $(\tmin,(\tmin+\tmax)/2)$. Therefore, the nonpositivity of $\lambda_\alpha$ implies that $\tmin<0$, and the last four boxes of lines~2 and~3 of Table~\ref{table:summary} are impossible. Similarly, the nonpositivity of $\tmin+\tmax$ yields $\lambda_\alpha<0$, hence two boxes of column~2 of Table~\ref{table:summary} are impossible.

Let us observe from Table~\ref{table:summary} that \eqref{eq}--\eqref{initial} may behave like, at least, five different classical reaction--diffusion equations.

\begin{enumerate}
\item When $\lambda_\alpha <0$ (line 2 of Table~\ref{table:summary}), then $\tmin<0$ and the zero steady state is linearly unstable. We are then facing a non-degenerate monostable\footnote{By a \lq\lq non-degenerate monostable situation'' we mean that the equation behaves like a standard  reaction-diffusion  equation with a monostable nonlinearity $f$ with non-zero slope at zero, a typical example being $f(\rho)=\rho(\rho-\theta_0)(1-\rho)$ with $\theta_0<0$.} situation: persistence occurs whatever the size of the initial density (hair trigger effect). This is proved in Theorem \ref{th:rescue}.
\item Whereas the critical case $\lambda _\alpha =0$ leads to extinction in classical Fisher-KPP equations (see e.g.~\cite{Berestycki_Ham_1} for such results in a periodic framework), \eqref{eq}--\eqref{initial} still enjoys the hair trigger effect when $\lambda _\alpha =0$ (line 3 of Table~\ref{table:summary} where, necessarily, $\tmin<0<\tmin+\tmax$). The reason is that~\eqref{eq} then ``escapes'' from the non-degenerate regime and ``switches'' to a (slightly) degenerate monostable situation, for which the hair trigger effect still holds.\footnote{By a \lq\lq (slightly) degenerate monostable situation, for which the hair trigger effect still holds'' we mean a reaction-diffusion equation with a monostable nonlinearity $f$ satisfying $f(\rho) \sim r \rho^{1+p}$ as $\rho\to 0$, for some $r>0$ and $1<1+p\leq 3$. A typical example is $f(\rho)=\rho(\rho-\theta_0)(1-\rho)$ with $\theta_0=0$.} This is also proved in Theorem \ref{th:rescue}.
\item When $\lambda_\alpha>0$ and the center $(\tmin+\tmax)/2$ of the interval $\Theta$ is smaller than, or equal to,~$1/2$ (columns 3, 4 and 5), we are typically facing a bistable situation, for which the outcome may be the extinction or the persistence of the population according to the initial density, as for local bistable reaction-diffusion equations admitting a traveling front invading the trivial state $0$.  The possibility of extinction is proved in Theorem \ref{th:ext2}. The possibility of persistence is proved in Theorem \ref{thm:poss_persis_subcrit} when $\tmin+\tmax<1$, while it is conjectured in the critical case $\tmin+\tmax=1$ (see subsection~\ref{sec:conj} below).
\item When $\lambda_\alpha >0$ and the center $(\tmin+\tmax)/2$ of the interval $\Theta$ is larger than~$1/2$ while $\tmin<1/2$ (column 6), we are typically facing a situation similar to that of local bistable reaction-diffusion equations admitting a traveling front that, according to the amplitude of $\alpha$, is retracting (meaning that the null state invades the nontrivial state), standing (meaning that connection from the null state to the nontrivial state is stationary), or possibly invading (meaning that the nontrivial state invades the null state). The possibility of extinction is proved in Theorem \ref{th:ext2}. The systematic extinction for $\alpha>\alpha^\star$ is proved in Theorem \ref{thm:poss_persis_supercrit}, while the possibility of persistence is conjectured when $\alpha \leq \alpha^\star$ (see subsection~\ref{sec:conj} below).
\item When $\lambda_\alpha>0$ and $\tmin\geq1/2$ (column 7), we are typically facing a bistable situation for which all solutions go extinct, as for local bistable reaction-diffusion equations admitting a standing or retracting front. This is proved in Theorem \ref{thm:syst-ext}.
\end{enumerate}

Our work effectively shows the possibility of evolutionary rescue in this model: an initial condition that would be, for instance, concentrated around $\theta_0>1/2$ should lead to extinction in the absence of mutations but might persist in the presence of mutations. As a matter of fact, it will automatically persist if, for instance, $\tmin+\tmax\leq 0$. However, we also observe that in the whole Table~\ref{table:summary}, the higher the mutation rate $\alpha$ is, the higher the chances of extinction are. This phenomenon is known as ``lethal mutagenesis" \cite{BulSan07}. There is therefore an interesting trade-off: evolutionary rescue is made possible by the presence of mutations but is made difficult by large mutation rates, which is consistent with the findings of \cite{AncLam19}. 

From a mathematical point of view, this paper is one of the first to provide rigorous results on reaction-diffusion equations taking into account diffusion in the spatial and the trait variables together with selection with respect to the trait and a nonlocal effect in the reaction. The fact that the trait variable is the Allee threshold is new (to the best of our knowledge) and mathematically challenging. The proofs include tools from nonlinear analysis, eigenvalue problems, variational arguments, pointwise comparison principles, and integral estimates.

\subsection{The \lq\lq E. or P. conjecture''}\label{sec:conj}

When $0\leq \tmin<1/2$ and $\tmin+\tmax\geq 1$, the situation is very intricate and seems to depend dramatically on the coefficient $\alpha$. When~$\alpha$ is above some threshold $\alpha^\star$, we prove systematic extinction in the case $\tmin+\tmax>1$. We also prove in Section~\ref{s:ext-pers} that extinction is always a possible outcome in the case $0\leq \tmin<1/2$ and $\tmin+\tmax\geq 1$ for some initial conditions. However, the possibility of defining $\alpha^\star$ as a ``sharp threshold" perfectly separating the ``E. or P.'' behavior (both extinction and persistence are possible depending on the initial condition)  when $\alpha\le \alpha^\star$ and the  ``E.'' behavior when $\alpha> \alpha^\star$ is not proved. Namely, for decreasing values of $\alpha$, the outcome could alternate between  ``E. or P.'' and ``E''. Nevertheless, in view of our numerical simulations (see Section~\ref{sec:numerics}), we conjecture that there exists sharp threshold $\alpha^\star$ and that in each one of the two yellow cells in Table~\ref{table:summary}, the outcome is ``E. or P.''. 
However, the possibility of persistence in these two cases remains an open question.

This means that:
\begin{enumerate}
\item we expect the critical case $\tmin+\tmax=1$ to be exactly similar to the sub-critical case $\tmin+\tmax<1$, $\tmin\geq 0$;
\item we expect that in the super-critical case $\tmin+\tmax>1$, $0\leq\tmin<1/2$, one can chose $\alpha^\star$ such that persistence is possible if and only if the mutation rate $\alpha$ is smaller than or equal to $\alpha^\star$. More precisely, we expect that when $\tmin+\tmax>1$, $0\leq\tmin<1/2$ and $\alpha\le\alpha^\star$, 
\textit{some} populations concentrate around $\tmin<1/2$ and by doing so escape extinction.
\end{enumerate}


\subsection{Generalizations of the model \eqref{eq} and open questions}

\noindent \textit{Spreading properties and traveling waves.} As mentioned above, we plan to analyze the spreading properties of the solutions of~\eqref{eq} in a future work. We may look for positive traveling waves of the form $u(t,x,\theta)=U(x-c\,t,\theta)$. This would also lead to traveling fronts for the total population  $\rho(t,x)=\int_{\Theta} U(x-c\,t ,\theta)\,\textup{d}\theta.$ Note that, necessarily, there would be a function $\omega$ such that the mean trait would satisfy $\tb(t,x)=\omega(x-c\,t)$. Thus, the equation satisfied by $\rho$ would take the general form:
\begin{equation}\label{eq:rho_bis}
\rho_t=d\rho _{xx}+f(\rho,x- c\,t).
\end{equation}
Several studies have investigated this type of equation when the function $f$ is known \cite{BerDie09,BerFan18,BouNad15,BouGil19,Ham97aa,Ham97bb}. These results cannot of course be applied as such, since here the function $f$ itself is not known, as it depends on $u$. Besides traditional existence and uniqueness results, one may study the asymptotic behavior of the traveling fronts at $\pm \infty$, the limit of the mean fitness at $\pm \infty$, the monotonicity of $U$ and of $\omega$  and the pulled nature (linear minimal speed $c^*=2\sqrt{-d\, \lambda_\alpha}$) or pushed nature (nonlinear minimal speed $c^*>2\sqrt{-d\, \lambda_\alpha}$) of the waves depending on the parameter values.

\noindent \textit{Evolutionary trade-offs.} With the model \eqref{eq} with $\tmax<1$, having higher values of $\theta$ is always disadvantageous. Not only the growth term becomes negative at low densities, but even the maximum per capita growth rate $\max_{\rho >0}(\rho-\theta)(1-\rho)=(1-\theta)^2/4$ is decreased. It seems however natural to consider that the need for cooperation between the individuals which is taken into account when $\theta$ is increased would also lead to a higher per capita growth rate. In other terms, there would be a trade-off between the trait $\theta$ and the maximum per capita growth rate. We propose the following extension of  \eqref{eq}: the population density $u=u(t,x,\y)$, is structured by an abstract phenotypic trait $\y\in\Omega \subset\R^k$ ($k\ge 1$), and satisfies:
\begin{equation}
\label{eq:extended}
u_t=d u_{xx}+\alpha \Delta_{\y} u +r(\y) \, u(\rho-\theta(\y))(1-\rho) \quad\text{for all } t>0,\ x\in \R,\ \y\in\Omega.
\end{equation}
Here, the Allee threshold depends on the phenotype via a function $\y\mapsto\theta(\y) \in (\tmin,\tmax)$. The trade-off is taken into account by assuming that the function $\tmax-\theta(\y)$ and the maximum per capita growth rate $r(\y) \, (1-\theta(\y))^2/4$ reach their maximum at different positions $\y_1$ and $\y_2$ in $\Omega$. With phenotypes around $\y_1$, the Allee threshold is low (cooperation between the individuals is not required) and the maximum per capita growth rate is also low. With phenotypes around $\y_2$, the Allee threshold is high (cooperation is  required), but also leads to a higher maximum per capita growth rate. With this approach, we conjecture that the population does not necessarily concentrate on trait values such that $\theta(\y)\approx \tmin$.

\subsection{Organization of the paper}

We start with some {\it a priori} estimates and the well-posedness of the Cauchy problem \eqref{eq}--\eqref{initial} in Section~\ref{s:preliminaries}. In Section~\ref{s:ext-pers} we prove all the extinction and persistence results of Table~\ref{table:summary} by combining Theorems~\ref{thm:syst-ext},~\ref{th:rescue},~\ref{th:ext2},~\ref{thm:poss_persis_subcrit}, and~\ref{thm:poss_persis_supercrit}. Lastly, in Section~\ref{sec:numerics} we present some numerical results supporting the aforementioned E.~or~P.~conjecture.

\section{Preliminaries}\label{s:preliminaries}

This section is devoted to the analysis of the Cauchy problem~\eqref{eq}--\eqref{initial}. Before doing so in Section~\ref{sec22}, we first derive in Section~\ref{sec21} some \textit{a priori} estimates and bounds for any classical solution of~\eqref{eq}--\eqref{initial}. 

\subsection{Global bounds and comparison between the population density $u$ and its mass $\rho$}\label{sec21}

In this section, we consider a classical solution $u\in\mathcal{C}^{1;2}_{t;(x,\theta)}((0,T^*)\times\R\times\overline{\Theta})\cap\mathcal{C}([0,T^*)\times\R\times\overline{\Theta})$ of~\eqref{eq}--\eqref{initial} in some time interval $[0,T^*)$ with $0<T^*\le +\infty$. We also assume that $u$ is locally bounded in time, that is, $u$ is bounded in $[0,T]\times\R\times\overline{\Theta}$ for every $T\in(0,T^*)$. The mass $\rho$ is then of class $\mathcal{C}^{1;2}_{t;x}((0,T^*)\times\R)\cap\mathcal{C}([0,T^*)\times\R)$ and it is locally bounded in time.

Let us first begin with the positivity of the population density $u$ and its mass $\rho$. For any $T\in(0,T^*)$, considering temporarily $\rho$ as a fixed function in $L^\infty([0,T]\times\R)$ and denoting $A(t,x,\theta)=(\rho(t,x)-\theta)(1-\rho(t,x))$, we find that the solution $u$ satisfies the equation~$u_t-d u_{xx} - \alpha u_{\theta\theta}=Au$ in $(0,T]\times\R\times\overline{\Theta}$, which is a local and linear parabolic equation with bounded space-time heterogeneous coefficients. Since $\underline{u}=0$ is a solution of this equation and since $u_0\geq\underline{u}$ with $u_0\not\equiv\underline{u}$ in $\R\times\overline{\Theta}$, we deduce from the parabolic maximum principle and Hopf lemma that
$$u(t,x,\theta)>0\ \hbox{ for all $(t,x,\theta)\in(0,T]\times\R\times\overline{\Theta}$}.$$
This implies in turn that $\rho(t,x)=\int_\Theta u(t,x,\theta)\,\textup{d}\theta>0$ for all $(t,x)\in(0,T]\times\R$. Finally, as $T$ is arbitrary in $(0,T^*)$, one gets that
$$u>0\hbox{ in }(0,T^*)\times\R\times\overline{\Theta},\ \hbox{ and }\ \rho>0\hbox{ in }(0,T^*)\times\R.$$

From the positivity of $u$ and $\rho$, we then easily derive the global boundedness of the mass $\rho$. To do so, we integrate equation \eqref{eq} over $\theta\in(\tmin,\tmax)$ and, using the no-flux boundary conditions~\eqref{neumann}, we reach
$$\rho_t=d\rho_{xx}+\left(\rho ^2-\int_{\Theta} \theta\, u(t,x,\theta)\,\textup{d}\theta\right)(1-\rho)\ \hbox{ for all $0<t<T^*$ and $x\in\R$}.$$
The previous equation can be rewritten as
\begin{equation}\label{eq-rho}
\rho_t=d\rho _{xx}+\rho\left(\rho-\tb\right)(1-\rho)\ \hbox{ for all $0<t<T^*$ and $x\in\R$},
\end{equation}
where
$$\tb(t,x):=\frac 1{\rho(t,x)}\int_{\Theta} \theta\,u(t,x,\theta)\,\textup{d}\theta$$
represents the mean trait at time $t\in(0,T^*)$ and spatial position $x\in \R$ (remember that~$\rho$ is pointwise positive in $(0,T^*)\times\R$). Since $u$ is pointwise positive, one has
\begin{equation}\label{tminmax}
\tmin<\tb(t,x)<\tmax<1\ \hbox{ for all  $t\in(0,T^*)$ and $x\in\R$}.
\end{equation}
Hence, together with~\eqref{initial}, the continuity of $\rho$ in $[0,T^*)\times\R$, and the comparison principle applied to~\eqref{eq-rho}, it follows that
\begin{equation} \label{eq:max_rho}
\sup_{(t,x)\in[0,T^*)\times\R}\rho(t,x)\leq \max (M,1).
\end{equation}

As a immediate consequence of~\eqref{eq:max_rho} and the positivity of $\rho$, the nonlinear term in~\eqref{eq} satisfies
\begin{equation}\label{boundf}
\vert u(\rho-\theta)(1-\rho)\vert \leq C\,u\ \hbox{ in }[0,T^*)\times\R\times\overline{\Theta}
\end{equation}
for some constant $C>0$. The maximum principle then implies that
\begin{equation}\label{boundu}
\|u(t,\cdot,\cdot)\|_{L^\infty(\R\times\overline{\Theta})}\le \textup{e}^{Ct}\|u_0\|_{L^\infty(\R\times\overline{\Theta})}\ \hbox{ for all }t\in[0,T^*).
\end{equation}
In particular, the solution $u$ is bounded if $T^*<+\infty$. On the other hand, if~$T^*=+\infty$, since the function $u$, which is then positive in $(0,+\infty)\times\R\times\overline{\Theta}$, solves a linear reaction-diffusion equation of the form
$$u_t=du_{xx}+\alpha u_{\theta \theta} +\phi(t,x)u,$$
with $\vert \phi (t,x)\vert \leq C$ by~\eqref{boundf}, and with Neumann boundary conditions on $(0,+\infty)\times\R\times\partial\Theta$, it follows from the standard Harnack inequality~\cite{KrySaf80,Mos-64} that there exists a constant $C'>0$ such that
$$u(t+1,x,\theta)\ge C'u(t,x',\theta')\ \hbox{ for all $t\ge1$, $\theta,\theta'\in\overline{\Theta}$, and $|x-x'|\le1$}.$$
Integrating the above inequality over $\theta \in \Theta$ and using the global boundedness~\eqref{eq:max_rho} of~$\rho$ (which holds whether $T^*$ be finite or not), one infers that $u(t,x',\theta')\leq\max(M,1)/C'$, and thus~$u$ is globally bounded too if $T^*=+\infty$. To sum up, $u$ is bounded in~$[0,T^*)\times\R\times\overline{\Theta}$, whether~$T^*$ be finite or not.

From~\eqref{boundf}, we also infer the limit of $u$ and $\rho$ at spatial infinity. Indeed, from the inequality $u_t\leq d u_{xx}+\alpha u_{\theta \theta}+Cu$ in $(0,T^*)\times\R\times\overline{\Theta}$ and the comparison principle, it follows that the nonnegative solution $u(t,x,\theta)$ is dominated from above by the nonnegative solution $v=v(t,x)$ of $v_t=d v_{xx}+Cv$ with initial condition $v_0$ defined by~$v_0(x):=\max_{\theta \in \overline \Theta} u_0(x,\theta)$ for all $x\in\R$. Thus, as $v_0\in C_c(\R,[0,+\infty))$, one infers that
\begin{equation}
\label{zero-infini}
\lim_{x\to \pm\infty}u(t,x,\theta)=0,\  \text{uniformly in $\theta\!\in\!\overline{\Theta}$, and locally uniformly in $t\!\in\![0,T^*)$}.
\end{equation}
As a consequence, one also gets that $\rho(t,x)\to0$ as $x\to\pm\infty$, locally uniformly with respect to~$t\in[0,T^*)$.

From~\eqref{boundf}, we can also reproduce the argument in \cite[Section 2]{BHR-17} to compare the population density $u$ and its mass $\rho$. For the sake of completeness, we briefly recall this argument. The inequalities $-Cu\leq u_t-d u_{xx}-\alpha u_{\theta \theta}\leq Cu$ imply
\begin{equation*}
u^{+}_t -d u^{+}_{xx}-\alpha u^{+}_{\theta\theta}\ge0\quad\text{and}\quad u^{-}_t -d u^{-}_{xx}-\alpha u^{-}_{\theta\theta}\le0\quad\hbox{in }(0,T^*)\times\R\times\overline{\Theta},
\end{equation*}
where
$$u^{\pm}(t,x,\theta):=\textup{e}^{\pm Ct}u(t,x,\theta).$$
Then, on the one hand, denoting $w[t]=w[t](\tau,x,\theta)$ the solution of the heat equation~$w_\tau=d w_{xx}+\alpha w_{\theta\theta}$ in $(0,+\infty)\times\R\times\overline{\Theta}$ with no-flux boundary conditions on~$(0,+\infty)\times\R\times\partial\Theta$ and with initial condition~$w[t](0,\cdot,\cdot):=u(t,\cdot,\cdot)$, it follows from the comparison principle that for every~$0\le\tau\le t<T^*$, $x\in\R$ and $\theta\in\overline{\Theta}$,
\begin{equation*}
w[t-\tau](\tau,x,\theta)\,\textup{e}^{-C\tau}\leq u(t,x,\theta)\leq w[t-\tau](\tau,x,\theta)\,\textup{e}^{C\tau}.
\end{equation*}
On the other hand, from the global boundedness of $u$ in $[0,T^*)\times\R\times\overline{\Theta}$ and the local-in-time Harnack inequality proved in \cite[Theorem~1.2]{BHR-17}\footnote{The proof can be straightforwardly extended to the cylindrical domain $\R\times\Theta$ considered here with the Neumann boundary conditions on $\R\times\partial\Theta$, see also~\cite[footnote~1]{BHR-17}.}, we deduce that, for every $\tau\in(0,T^*)$ and $p>1$, there exists a constant $\widetilde{C}_{p,\tau}>0$ (which also depends on~$d$,~$\alpha$ and~$\|u\|_{L^\infty([0,T^*)\times\R\times\overline{\Theta})}$) such that
\begin{equation*}
\frac{\left( w[t-\tau](\tau,x',\theta') \right)^{p}}{\widetilde{C}_{p,\tau}^p}\leq w[t-\tau](\tau,x,\theta)\leq\widetilde{C}_{p,\tau} \left( w[t-\tau](\tau,x',\theta') \right)^{1/p}
\end{equation*}
for all $t\in[\tau,T^*)$, $\theta,\theta'\in\overline{\Theta}$ and $|x-x'|\le1$ (notice that the left and right inequalities in the above formula are actually equivalent since~$\theta$,~$\theta'$ are arbitrary in $\overline{\Theta}$ and $x$ and $x'$ are arbitrary real numbers such that $|x-x'|\le1$). We deduce that, for every $\tau\in(0,T^*)$ and $p>1$, there exists a constant $C_{p,\tau}>0$ such that
\begin{equation}
\label{ineg0}
\frac{u^p(t,x,\theta)}{C_{p,\tau}}\leq \rho(t,x)\leq C_{p,\tau}\,u^{1/p}(t,x,\theta)\quad\text{for all }(t,x,\theta)\in[\tau,T^*)\times\R\times\overline{\Theta}, 
\end{equation}
which leads to
\begin{equation}
\label{ineg}
\frac{\rho^p(t,x)}{C_{p,\tau}^p}\leq u(t,x,\theta)\leq C_{p,\tau}^{1/p}\rho^{1/p}(t,x)\quad\text{for all }(t,x,\theta)\in[\tau,T^*)\times\R\times\overline{\Theta}. 
\end{equation}

\begin{rem}The comparison \eqref{ineg0}, or \eqref{ineg}, between $u$ and $\rho$ is the main estimate of this subsection, and will be useful in several other parts of this paper, in particular in Section~\ref{s:ext-pers}.
\end{rem}

We finally derive an explicit upper bound for the mass $\rho$ at large time if $T^*=+\infty$. To do so, observe first that, whether $\rho(t,x)$ be smaller than $1$, equal to $1$, or larger than~$1$, one has
$$(\rho(t,x)-\overline{\theta}(t,x))\,(1-\rho(t,x))\le(1-\overline{\theta}(t,x))\,(1-\rho(t,x))\hbox{ for all $(t,x)\in(0,T^*)\times\R$}.$$
From~\eqref{eq-rho} and the positivity of $\rho$ in $(0,T^*)\times\R$, one gets that
\begin{equation*}
    \rho_t-d \rho_{xx}\leq\rho\,(1-\overline{\theta})\,(1-\rho)\quad\text{in }(0,T^*)\times\R.
\end{equation*}
One also recalls that $\tmin<\overline{\theta}(t,x)<\tmax<1$ for all $t\in(0,T^*)$ and $x\in\R$. Therefore,
$$\rho(t,x)\,(1-\overline{\theta}(t,x))\,(1-\rho(t,x))\le\left\{\begin{array}{ll}
\!\!(1-\tmin)\,\rho(t,x)\,(1-\rho(t,x)) & \!\!\hbox{if }0<\rho(t,x)\le1,\vspace{3pt}\\
\!\!(1-\tmax)\,\rho(t,x)\,(1-\rho(t,x)) & \!\!\hbox{if }\rho(t,x)>1.\end{array}\right.$$
If $T^*=+\infty$, by comparison with a classical Fisher--KPP type equation, it then follows that
\begin{equation}
    \label{eq:limsup-mass}
    \limsup_{t\to +\infty}\Big(\sup_{x\in\R}\rho(t,x)\Big)\leq 1 \ \hbox{ (if $T^*=+\infty$)},
\end{equation}
hence, together with~\eqref{ineg},
\begin{equation}
    \label{eq:limsup-density}
    \limsup_{t\to +\infty}\Big(\sup_{(x,\theta)\in\R\times\overline{\Theta}}u(t,x,\theta)\Big)\leq C_{p,\tau}^{1/p}\ \hbox{ (if $T^*=+\infty$)},
\end{equation}
for every $\tau>0$ and $p>1$.

\begin{rem}
Let us point out that similar upper bounds on $u$  could also be deduced from~\eqref{eq:max_rho},~\eqref{eq:limsup-mass} and~\cite[Proposition 2.3]{Turanova_2015}. In both cases though, the upper bound for $u$ depends on $d$ and $\alpha$.
\end{rem}

\subsection{The well-posedness of the Cauchy problem~\eqref{eq}--\eqref{initial}}\label{sec22}

Several arguments used in the forthcoming sections require a refined knowledge of the functional space which the solution $u$ belongs to. Therefore, as a mandatory preliminary, we study the well-posedness of the Cauchy problem \eqref{eq}--\eqref{initial}.

\begin{rem}
Hereafter, by a solution of~\eqref{eq}--\eqref{initial}, we mean a solution in Tikhonov's uniqueness class, that is, a solution $u\in\mathcal{C}^{1;2}_{t;(x,\theta)}((0,+\infty)\times\R\times\overline{\Theta})\cap\mathcal{C}([0,+\infty)\times\R\times\overline{\Theta})$ for which, for every $T>0$, there exists a constant $A_T>0$ such that $u(t,x,\theta)=o(e^{A_T|x|^2})$ as $x\to\pm\infty$, uniformly in $(t,\theta)\in[0,T]\times\overline{\Theta}$. Indeed, without any such restriction on the growth at infinity, solutions of reaction-diffusion Cauchy problems may not be unique, see, \textit{e.g.} \cite[Chapter 9]{Smo-book}.
\end{rem}

\begin{prop}[Well-posedness]\label{prop:well-posedness} Let
$$X=\big\{\varphi:\mathbb{R}\times\overline{\Theta}\to\R:\varphi\hbox{ is bounded and uniformly continuous in }\mathbb{R}\times\overline{\Theta}\big\}$$ 
endowed with the usual sup norm, denoted by $\|\ \|_X$. Then the Cauchy problem \eqref{eq}--\eqref{initial} admits a unique solution $u$ such that
\begin{equation*}
     t\mapsto u(t,\cdot,\cdot)\in \mathcal{C}\left([0,+\infty),X\right)\cap\mathcal{C}^1\left((0,+\infty),X\right).
\end{equation*}
Furthermore, $u>0$ in $(0,+\infty)\times\R\times\overline{\Theta}$.
\end{prop}

The basic idea is, as usual, to first prove the existence and uniqueness of a local-in-time solution by a fixed point argument and then to deduce the existence and uniqueness of a global solution from an \textit{a priori} $L^\infty$ estimate. As such an estimate is proved above in Subsection~\ref{sec21}, we first focus on the local well-posedness. In the sequel, we denote
$$|\Theta|=\tmax-\tmin\ \hbox{ and }\ \theta_m=\max(|\tmax|,|\tmin|).$$

\begin{lem}[Local well-posedness]\label{lem:local-well-posedness} 
For any $\tau\geq 0$, any $K>0$ and any $u_\tau\in X$ with $0\le u_\tau\le K$ in $\R\times\overline{\Theta}$, the following problem:
\begin{equation}
    \label{sys:local-in-time}
    \begin{cases}
        u_t=du_{xx}+\alpha u_{\theta \theta}+u(\rho-\theta)(1-\rho), & t\in(\tau,\tau+T_K],\ x\in \R,\ \theta\in\Theta,\vspace{3pt}\\
        u_\theta(t,x,\tmin)=u_\theta(t,x,\tmax)=0, & t\in(\tau,\tau+T_K],\ x\in \R,\vspace{3pt}\\
        u(\tau,x,\theta)=u_\tau(x,\theta), & x\in \R,\ \theta\in\Theta,
    \end{cases}
\end{equation}
where
\begin{equation}
\label{defTK}
T_K=\frac{1}{3\,(2K|\Theta|+\theta_m)\,(2K|\Theta|+1)}>0,
\end{equation}
admits a unique classical $\mathcal{C}^{1;2}_{t;(x,\theta)}((\tau,\tau+T_K]\times\R\times\overline{\Theta})\cap\mathcal{C}([\tau,\tau+T_K]\times\R\times\overline{\Theta})$ solution~$u$ such that 
\begin{equation}\label{cont}
     t\mapsto u(t,\cdot,\cdot)\in \mathcal{C}\left([\tau,\tau+T_K],X\right)\cap\mathcal{C}^1\left((\tau,\tau+T_K],X\right).
\end{equation}
Furthermore, $u\ge0$ in $[\tau,\tau+T_K]\times\R\times\overline{\Theta}$.
\end{lem}

\begin{proof}
Define, for $T>0$ and $C>0$, the sets
\begin{equation*}
    X_C = \{ v\in X : \|v\|_X \leq C\}\ \hbox{ and }\ \mathcal{X}_{T,C}=\mathcal{C}\left([\tau,\tau+T|,X_C\right).
\end{equation*}
The set $\mathcal{X}_{T,C}$, endowed with the distance induced by the norm  $\|v\|:=\max_{t\in[\tau,\tau+T]}\|v(t)\|_X$ for $v\in\mathcal{C}([\tau,\tau+T],X)$, is a complete metric space. The elements of~$\mathcal{X}_{T,C}$ can also be considered with a slight abuse of notation as functions of the variables $(t,x,\theta)\in[\tau,\tau+T]\times\R\times\overline{\Theta}$. 

Let $\Phi:\mathcal{X}_{T,C}\mapsto \mathcal{C}\left([\tau,\tau+T],X\right)$ be the mapping that associates to $v\in\mathcal{X}_{T,C}$ the unique mild solution $u=\Phi[v]\in\mathcal{C}([\tau,\tau+T],X)$ of 
\begin{equation*}
    \begin{cases}
        \displaystyle u_t=du_{xx}+\alpha u_{\theta \theta}+u\Big(\int_{\Theta}v-\theta\Big)\,\Big(1-\int_{\Theta}v\Big), & t\in(\tau,\tau+T],\ x\in \R,\  \theta\in\Theta,\vspace{3pt}\\
        u_\theta(t,x,\tmin)=u_\theta(t,x,\tmax)=0, & t\in(\tau,\tau+T],\ x\in \R,\vspace{3pt}\\
        u(\tau,x,\theta)=u_\tau(x,\theta), & x\in \R,\ \theta\in\Theta.
    \end{cases}
\end{equation*}
The solution $u$ is indeed well-defined, as the above problem is just a linear parabolic Cauchy problem: more precisely, letting
$$\rho_v(t,x)=\int_\Theta v(t,x,\theta)\,\textup{d}\theta$$
for $(t,x)\in[\tau,\tau+T]\times\R$, and $G=G(t,x,\theta;s,y,\eta)$ be the Green function associated with the parabolic operator $\partial_t-d\partial_{xx}-\alpha\partial_{\theta\theta}$ in the spatial domain $\R\times\Theta$ with  Neumann boundary conditions on $\R\times\partial\Theta$, one has, for every $(t,x,\theta)\in(\tau,\tau+T]\times\R\times\overline{\Theta}$,
\begin{align*}
    u(t,x,\theta)=\Phi[v](t,x,\theta) & =\int_{\R\times\Theta} G(t,x,\theta;\tau,y,\eta)\,u_\tau(y,\eta)\,\textup{d}y\,\textup{d}\eta \\
    & \quad + \int_\tau^t \int_{\R\times\Theta}G(t,x,\theta;s,y,\eta)\,\Phi[v](s,y,\eta)\times\\
    & \qquad\qquad\qquad\qquad\times(\rho_v(s,y) -\eta)\,(1-\rho_v(s,y))\, \textup{d}y\,\textup{d}\eta\,\textup{d}s.
\end{align*}
We aim at showing that $\Phi$ is a contraction mapping from $\mathcal{X}_{T,C}$ into itself when the positive parameters~$T$ and $C$ are appropriately chosen, so that it admits a unique fixed point. 

To do so, first of all, observe that, for any $v\in\mathcal{X}_{T,C}$, the inequality
$$|(\rho_v(s,y)-\eta)(1-\rho_v(s,y))|\!\leq\! (|\Theta|\|v\|+\theta_m)\,(|\Theta|\|v\|+1),\hbox{ for all $(s,y,\eta)\!\in\![\tau,\tau+T]\!\times\!\R\!\times\!\overline{\Theta}$},$$
together with $\|G(t,x,\theta;s,\cdot,\cdot)\|_{L^1(\R\times\Theta)}=1$ for all $t>s$ and $(x,\theta)\in\R\times\overline{\Theta}$, yields
\begin{align*}
    \|\Phi[v]\| & \leq K + T\,(|\Theta|\|v\|+\theta_m)\,(|\Theta|\|v\|+1)\|\Phi[v]\|\vspace{3pt}\\
    & \leq K+ T\,(C|\Theta|+\theta_m)\,(C|\Theta|+1)\|\Phi[v]\|.
\end{align*}
Assuming that $T>0$ is so small that 
\begin{equation}
    \label{small_T}
    T\,(C|\Theta|+\theta_m)\,(C|\Theta|+1)<1, 
\end{equation}
we deduce
\begin{equation*}
    \|\Phi[v]\|\leq \frac{K}{1-T\,(C|\Theta|+\theta_m)\,(C|\Theta|+1)},
\end{equation*}
so that $\Phi$ maps $\mathcal{X}_{T,C}$ into itself as soon as~\eqref{small_T} is fulfilled together with 
\begin{equation}
    \label{small_C}
    \frac{K}{1-T\,(C|\Theta|+\theta_m)\,(C|\Theta|+1)}\le C.
\end{equation}

Now, assume that the conditions~\eqref{small_T}--\eqref{small_C} are indeed satisfied, so that $\Phi(\mathcal{X}_{T,C})\subset\mathcal{X}_{T,C}$. Let $v,w\in\mathcal{X}_{T,C}$, and denote $\rho_v=\int_\Theta v(\cdot,\cdot,\theta)\,\textup{d}\theta$ and $\rho_w=\int_\Theta w(\cdot,\cdot,\theta)\,\textup{d}\theta$. After some straightforward calculations, we find that $z=\Phi[v]-\Phi[w]\in\mathcal{C}([\tau,\tau+T],X)$ is a mild solution of
\begin{equation*}
    z_t-dz_{xx}-\alpha z_{\theta\theta}= z\left(\rho_v -\theta\right)\left(1-\rho_v\right)+ \Phi[w]\left( 1+\theta-\rho_v-\rho_w \right)(\rho_v-\rho_w)
\end{equation*}
in $(\tau,\tau+T]\times\R\times\Theta$ with Neumann boundary conditions on $(\tau,\tau+T]\times\R\times\partial\Theta$. Using similarly the Green's function $G$ and $\|G(t,x,\theta;s,\cdot,\cdot)\|_{L^1(\R\times\Theta)}=1$ for all $t>s$ and $(x,\theta)\in\R\times\overline{\Theta}$, we find that
\begin{equation*}
    \|z\|\leq T\,(C|\Theta|+\theta_m)\,(C|\Theta|+1)\,\|z\|+T\,C\,(1+\theta_m+2C|\Theta|)\,|\Theta|\,\|v-w\|,
\end{equation*}
so that
\begin{equation*}
    \|z\|=\|\Phi[v]-\Phi[w]\|\leq \frac{T\,C\,(1+\theta_m+2C|\Theta|)\,|\Theta|}{1-T\,(C|\Theta|+\theta_m)\,(C|\Theta|+1)}\,\|v-w\|.
\end{equation*}
Therefore $\Phi$ is a contraction mapping as soon as 
\begin{equation}
    \label{contractance}
    \frac{T\,C\,(1+\theta_m+2C|\Theta|)\,|\Theta|}{1-T\,(C|\Theta|+\theta_m)\,(C|\Theta|+1)}<1.
\end{equation}

One can easily check that the conditions \eqref{small_T}--\eqref{contractance} are compatible: for instance one may choose
$$C=\frac{3K}{2}>0, \quad  T=\frac{1}{3\,(C|\Theta|+\theta_m)\,(C|\Theta|+1)}>0.$$
As a consequence, by virtue of the Banach fixed point theorem, $\Phi$ admits a unique fixed point $v\in\mathcal{X}_{T,C}$ for the above choice of $T$ and $C$. By standard parabolic estimates, $v$ is then a classical $\mathcal{C}^{1;2}_{t;(x,\theta)}((\tau,\tau+T]\times\R\times\overline{\Theta})\cap\mathcal{C}([\tau,\tau+T]\times\R\times\overline{\Theta})$ solution of~\eqref{sys:local-in-time} (the continuity of $v$ in $[\tau,\tau+T]\times\R\times\overline{\Theta}$ is actually automatic by construction), and the map~$t\mapsto v(t,\cdot,\cdot)$ belongs to  $\mathcal{C}([\tau,\tau+T],X)\cap\mathcal{C}^1((\tau,\tau+T],X)$.

Furthermore, by picking now
$$\widetilde{C}=2K$$
and
$$T_K=\frac{1}{3\,(2K|\Theta|+\theta_m)\,(2K|\Theta|+1)}\in(0,T)$$
as in~\eqref{defTK}, it follows as above that $\Phi$ admits a unique fixed point $u$ in $\mathcal{X}_{T_K,2K}$, which is a classical $\mathcal{C}^{1;2}_{t;(x,\theta)}((\tau,\tau+T_K]\times\R\times\overline{\Theta})\cap\mathcal{C}([\tau,\tau+T_K]\times\R\times\overline{\Theta})$ solution of~\eqref{sys:local-in-time}, and the map $t\mapsto u(t,\cdot,\cdot)$ belongs to $\mathcal{C}([\tau,\tau+T_K],X)\cap\mathcal{C}^1((\tau,\tau+T_K],X)$. Since the function $v$ restricted to $[\tau,\tau+T_K]\times\R\times\overline{\Theta}$ solves the same problem as $u$ and since $v(t,\cdot,\cdot)\in\mathcal{X}_{T_K,C}\subset\mathcal{X}_{T_K,2K}$ for every $t\in[\tau,\tau+T]\supset[\tau,\tau+T_K]$, it follows by uniqueness that $u\equiv v_{|[\tau,\tau+T_K]\times\R\times\overline{\Theta}}$. On the other hand, since the left-hand sides of the inequalities~\eqref{small_T}--\eqref{contractance} are increasing with respect to $T\in[0,T_K]$, one infers that, for any $T'\in(0,T_K]$, $\Phi$ has a unique fixed point in $\mathcal{X}_{T',2K}$, and that this unique fixed point is nothing but the restriction of~$u$ and~$v$ in $[\tau,\tau+T']\times\R\times\overline{\Theta}$.

Finally, consider any mild solution $U$ of~\eqref{sys:local-in-time} with
$$t\mapsto U(t,\cdot,\cdot)\in\mathcal{C}([\tau,\tau+T_K],X)\cap\mathcal{C}^1((\tau,\tau+T_K],X)$$
($U$ is then also a classical $\mathcal{C}^{1;2}_{t;(x,\theta)}((\tau,\tau+T_K]\times\R\times\overline{\Theta})\cap\mathcal{C}([\tau,\tau+T_K]\times\R\times\overline{\Theta})$ solution). We claim that
$$U\equiv u\ \hbox{ in }[\tau,\tau+T_K]\times\R\times\overline{\Theta}.$$
Indeed, first of all, since $0\le U(\tau,\cdot,\cdot)=u_\tau\le K$ in $\R\times\overline{\Theta}$, there is by continuity a maximal time $T_U\in(0,T_K]$ such that $\|U(t,\cdot,\cdot)\|_X<2K$ for all $t\in[0,T_U)$, and $\|U(T_U,\cdot,\cdot)\|_X=2K$ if $T_U<T_K$. By uniqueness, one gets that $U\equiv u\equiv v$ in $[\tau,\tau+T']\times\R\times\overline{\Theta}$ for every $T'\in(0,T_U)$, and then in $[\tau,\tau+T_U]\times\R\times\overline{\Theta}$ by continuity. But since $\|v(t,\cdot,\cdot)\|_X\le C=3K/2<2K$ for all $t\in[\tau,\tau+T]\supset[\tau,\tau+T_K]$, it follows that $\|U(T_U,\cdot,\cdot)\|_X\le3K/2<2K$. Hence $T_U=T_K$ and $U\equiv u$ in $[\tau,\tau+T_K]\times\R\times\overline{\Theta}$. Therefore, the constructed solution~$u$ is the unique mild and classical solution of~\eqref{sys:local-in-time}--\eqref{cont}.

Lastly, the nonnegativity of $u$ in $[0,T_K]\times\R\times\overline{\Theta}$ follows from the nonnegativity of $u_\tau$ in $\R\times\overline{\Theta}$ and from the comparison of $u$ with the trivial solution $0$, as in Section~\ref{sec21}. The proof of Lemma~\ref{lem:local-well-posedness} is thereby complete.
\end{proof}

\begin{proof}[Proof of Proposition $\ref{prop:well-posedness}$]
From Lemma~\ref{lem:local-well-posedness}, there exist a maximal existence time $T^*\in(0,+\infty]$ and a unique mild and classical solution $u\in \mathcal{C}^{1;2}_{t;(x,\theta)}((0,T^*)\times\R\times\overline{\Theta})\cap \mathcal{C}([0,T^*)\times\R\times\overline{\Theta})$ of~\eqref{eq}--\eqref{initial} such that $t\mapsto u(t,\cdot,\cdot)\in\mathcal{C}([0,T^*),X)\cap\mathcal{C}^1((0,T^*),X)$ (in particular,~$u$ is locally bounded in time in $[0,T^*)\times\R\times\overline{\Theta}$). Furthermore, $u\ge0$ in~$[0,T^*)\times\R\times\overline{\Theta}$. Lemma~\ref{lem:local-well-posedness} and the quantitative estimate~\eqref{defTK} in terms of $K$ also imply that
$$\|u(t,\cdot,\cdot)\|_X=\|u(t,\cdot,\cdot)\|_{L^\infty(\R\times\Theta)}\to+\infty\hbox{ as }t\mathop{\to}^<T^*\ \hbox{ if }T^*<+\infty.$$
But the classical solution $u$ is necessarily globally bounded in $[0,T^*)\times\R\times\overline{\Theta}$, from the arguments of Section~\ref{sec21}. Therefore, $T^*=+\infty$ and $u$ satisfies all desired properties stated in Proposition~\ref{prop:well-posedness}, including its positivity in $(0,+\infty)\times\R\times\overline{\Theta}$ by Section~\ref{sec21}. The proof of Proposition~\ref{prop:well-posedness} is thereby complete.
\end{proof}

\section{Persistence versus extinction}\label{s:ext-pers}

In this section, we consider the Cauchy problem \eqref{eq}--\eqref{initial} and we figure out whether solutions are persistent or go extinct in long time. Hereafter, extinction is defined as
\begin{equation*}
    \lim_{t\to+\infty}\Big(\sup_{(x,\theta)\in\mathbb{R}\times\overline{\Theta}}u(t,x,\theta)\Big)=0
\end{equation*}
and persistence is the opposite statement, namely
\begin{equation*}
    \limsup_{t\to+\infty}\Big(\sup_{(x,\theta)\in\mathbb{R}\times\overline{\Theta}}u(t,x,\theta)\Big)>0.
\end{equation*}
From the comparison~\eqref{ineg} (with here $T^*=+\infty$, any $\tau>0$ and any $p>1$), these two properties are equivalent to the similar ones obtained by replacing $\sup_{(x,\theta)\in\R\times\overline{\Theta}}u(t,x,\theta)$ with $\sup_{x\in\R}\rho(t,x)$.

It turns out that there are several different answers to the question of persistence or extinction, according to the trait range $\Theta=(\tmin,\tmax)$, to the initial conditions, and to the principal eigenvalue $\lambda_\alpha$ of the linearized operator around the trivial state~$0$ in the trait variables. We start by studying the properties of this eigenvalue $\lambda_\alpha$ in Section~\ref{sec:ev}. Then we prove in Section~\ref{sec:imposs_persis} the systematic (independently of the initial conditions $u_0$ satisfying~\eqref{initial}) extinction when $\tmin\ge1/2$ (column~7 of Table~\ref{table:summary}), and in Section~\ref{sec:imposs_extin} the systematic persistence when $\lambda_\alpha\le0$ (lines~2 and~3 of Table~\ref{table:summary}). We then show in Section~\ref{sec:poss_extin} the possibility of extinction when $\lambda_\alpha>0$ for small initial data (columns~3,~4,~5,~6 and~7 of Table~\ref{table:summary}), and we discuss in Section~\ref{sec:poss_persis} the possibility of persistence when~$\lambda_\alpha>0$ and~$\tmin<1/2$ (columns~3,~4,~5 and~6 of Table~\ref{table:summary}).

\subsection{A principal eigenvalue problem}\label{sec:ev}

For a given $\alpha>0$, we consider the Neumann principal eigenproblem
\begin{equation}
\label{vp}
\begin{cases}
-\alpha \varphi''+\theta\varphi=\lambda \varphi & \text{ in } \overline{\Theta},\vspace{3pt}\\
\varphi'(\tmin)=\varphi'(\tmax)=0,\vspace{3pt}\\
\varphi>0& \text{ in }\overline{\Theta},\vspace{3pt}
\end{cases}
\end{equation}
and denote $(\lambda_\alpha,\varphi_\alpha)$ the principal eigenpair, normalized by $\max_{\overline{\Theta}}\varphi_\alpha=1$. We have the variational formula
\begin{equation}
\label{rayleigh}
\lambda_\alpha=\min \left\{ Q_\alpha(\varphi):=\int_{\Theta}\big(\alpha \varphi'^2(\theta)+\theta \varphi^2(\theta)\big)\, \textup{d}\theta: \varphi \in E \right\},
\end{equation}
with
$$E:=\{\varphi \in H^1(\Theta):\|\varphi\|_{L^2(\Theta)}=1\}.$$

\begin{lem}[On the principal eigenvalue $\lambda _\alpha$]\label{lem:vp} The principal eigenvalue $\lambda_\alpha$ enjoys the following properties:
\begin{enumerate}[label=(\roman*)]
\item for all $\alpha>0$,
$$\tmin<\lambda _\alpha<\frac{\tmin+\tmax}{2};$$
\item the function $\alpha \mapsto \lambda _\alpha$ is increasing and concave in $(0,+\infty)$, and
$$\lim_{\alpha \to 0} \lambda_\alpha=\tmin, \quad \lim_{\alpha\to+\infty}\lambda_\alpha=\frac{\tmin+\tmax}{2};$$
\item if
$$\frac32\,\Big(\frac{\pi^2\alpha}{2}\Big)^{1/3}+\tmin\leq 0\ \hbox{ and }\ \Big(\frac{\pi^2\alpha}{2}\Big)^{1/3}+\tmin\leq\tmax,$$
then
$$\lambda_\alpha \leq \frac 32 \left(\frac{\pi ^2\alpha}2 \right)^{1/3}+\tmin.$$
\end{enumerate}
\end{lem}

\begin{proof} In $(ii)$, the fact that the function $\alpha\mapsto\lambda_\alpha$ is increasing in $(0,+\infty)$ and the determination of its limits as $\alpha\to0$ and $\alpha\to+\infty$ are classical properties (see e.g.~\cite[Lemma~2.1]{Hut-95}), that also yield~$(i)$. Additionally, since $\lambda_\alpha=\min_{\varphi\in E}Q_\alpha(\varphi)$ for each $\alpha>0$ and since the map $\alpha\mapsto Q_\alpha(\varphi)$ is linear (hence concave) in $(0,+\infty)$ for each $\varphi\in E$, one gets that the map $\alpha\mapsto \lambda_\alpha$ is concave in $(0,+\infty)$, and therefore continuous.

Let us now turn to the proof of~$(iii)$, which is slightly more subtle. Notice first that the condition $(3/2)\times(\pi ^2\alpha/2)^{1/3}+\tmin\leq 0$ implies that $\tmin<0$. We shall use the solution of a related eigenproblem (with constant coefficients) on an interval $[\tmin,\tmin+\eta]$, whose size $\eta\in(0,-\tmin]$ is to be optimized. Precisely, observe that the Neumann-Dirichlet principal eigenproblem
$$\begin{cases}
-\alpha \varphi''+(\tmin+\eta)\varphi=\lambda \varphi & \text{ in }[\tmin,\tmin+\eta],\vspace{3pt}\\
\varphi'(\tmin)=0,\ \ \varphi(\tmin+\eta)=0,\vspace{3pt}\\
\varphi>0& \text{ in }[\tmin,\tmin+\eta),\vspace{3pt}
\end{cases}$$
is explicitly solved as 
$$\varphi(\theta)=\varphi_\eta(\theta)=\sin\left(\frac{\pi}{2\eta}(\tmin+\eta-\theta)\right)\hbox{ for }\theta\in[\tmin,\tmin+\eta],\ \  \lambda=\widetilde{\lambda}_\eta=\frac{\pi^2\alpha}{4\eta^2}+\tmin+\eta.$$
Notice that the infimum $\inf_{0<\eta\leq -\tmin}\widetilde{\lambda}_ \eta$ is reached by choosing
$$\eta=\eta_{\textup{opt}}:=\left(\frac{\pi ^2\alpha}2 \right)^{1/3}\in(0,-\tmin),$$
and is equal to
$$\inf_{0<\eta\le-\tmin}\widetilde{\lambda}_\eta=\frac32\,\Big(\frac{\pi ^2\alpha}{2}\Big)^{1/3}+\tmin=:\widetilde{\lambda}_{\textup{opt}},$$
which is nonpositive by assumption. We denote $(\widetilde{\lambda}_{\textup{opt}},\varphi _{\textup{opt}})$ the eigenpair associated with $\eta=\eta_{\textup{opt}}$ and with $\|\varphi_{\textup{opt}}\|_{L^2(\tmin,\tmin+\eta_{\textup{opt}})}=1$. We use $\varphi_{\textup{opt}}$ (extended by zero in~$(\tmin+\eta_{\textup{opt}},\tmax]$, notice that $\tmin+\eta_{\textup{opt}}\le\tmax$ by assumption) as a test function in~$E$ and obtain
\begin{eqnarray*}
Q_\alpha(\varphi_{\textup{opt}})&=&\int_{\tmin}^{\tmin+\eta_{\textup{opt}}}\big(\alpha \varphi_{\textup{opt}}'^2(\theta)+\theta \varphi_{\textup{opt}}^2(\theta)\big)\,\textup{d}\theta\\
&\leq& \int_{\tmin}^{\tmin+\eta_{\textup{opt}}}\big(\alpha \varphi_{\textup{opt}}'^2(\theta)+(\tmin+\eta_{\textup{opt}}) \varphi_{\textup{opt}}^2(\theta)\big)\,\textup{d}\theta=\widetilde{\lambda}_{\textup{opt}}
\end{eqnarray*}
and thus $\lambda_\alpha \leq\widetilde{\lambda}_{\textup{opt}}=(3/2)\times(\pi^2\alpha/2)^{1/3}+\tmin$. The proof of Lemma~\ref{lem:vp} is thereby complete.
\end{proof}

\begin{lem}[On the principal eigenfunction $\varphi _\alpha$]\label{lem:fctopropre} For each $\alpha>0$, the principal eigenfunction $\varphi_\alpha$ of~\eqref{vp} is decreasing in $[\tmin,\tmax]$, strictly concave in $[\tmin,\lambda_\alpha)$, and strictly convex in $(\lambda_\alpha,\tmax]$.
\end{lem}

\begin{proof} Remember first that $\tmin<\lambda_\alpha<(\tmin+\tmax)/2<\tmax$ from Lemma~\ref{lem:vp}. The strict concavity/convexity properties follow from the equation $\alpha \varphi_\alpha''(\theta)=(\theta-\lambda_\alpha)\varphi_\alpha(\theta)$ and the positivity of $\varphi_\alpha$ in $[\tmin,\tmax]$. Next, $\varphi_\alpha'(\tmin)=0$ and the strict concavity in $[\tmin,\lambda_\alpha)$ enforce $\varphi_\alpha$ to decrease on this interval and then in $[\tmin,\lambda_\alpha]$. A similar argument applies in $[\lambda_\alpha,\tmax]$.
\end{proof}

To complete this section, let us observe that Lemma~\ref{lem:vp} gives all the impossible boxes in Table~\ref{table:summary}, whereas all other boxes are truly possible. The outcome of the solutions of~\eqref{eq}--\eqref{initial} in the possible boxes may or may not depend on the initial data and on the parameter $\alpha$, as we are going to see in the next subsections.

\subsection{$\tmin\ge 1/2$ makes persistence impossible}\label{sec:imposs_persis}

When $\tmin\geq1/2$ (column 7 of Table~\ref{table:summary}), no population can escape from extinction, as the following result shows.

\begin{thm}[Systematic extinction]\label{thm:syst-ext} If $\tmin\geq1/2$, then all solutions of the Cauchy problem \eqref{eq}--\eqref{initial} go extinct.
\end{thm}

\begin{proof} We begin with the simpler and more telling case
$$\tmin>\frac12.$$ 
For any $\ep> 0$, define
$$\overline{f}^{\ep}(s):=s\,\Big(s-\Big(\frac12 +\ep\Big)\Big)\,(1+\ep-s)$$
and observe that, as $\ep \to 0$, 
\begin{equation*}
\int _0^{1+\ep} \overline{f}^{\ep}(s)ds=-\frac \ep {12}+o(\ep).
\end{equation*}
We then fix $\ep>0$ small enough so that $\tmin>1/2 +2\ep$ and $\int _0^{1+\ep} \overline{f}^{\ep}(s)ds<0$. Consider now any solution $u$ of~\eqref{eq}--\eqref{initial}. From~\eqref{eq:limsup-mass} and the positivity of $\rho$ in $(0,+\infty)\times\R$, there exists $T_\ep>0$ such that $0<\rho(t,x)<1+\ep$ for all $t\geq T_\ep$ and $x\in \R$. Then, recalling that $\overline \theta (t,x)\in(\tmin,\tmax)\subset(1/2+2\ep,1)$, we claim that
\begin{equation*}
\rho(t,x)\,(\rho(t,x) -\overline \theta (t,x))\,(1-\rho(t,x))< \overline{f}^{\ep}(\rho(t,x))\ \ \text{for all }t\geq T_\ep,\ x\in\R.
\end{equation*}
Indeed, for every $(t,x)\in[T_\ep,+\infty)\times\R$, the quadratic polynomial function
$$s\mapsto\overline{f}^\ep(s)-s\,(s-\overline{\theta}(t,x))\,(1-s)$$
vanishes at $s=0$, is positive at $s=1+\ep$, and is negative at large~$s$ since $\tb(t,x)>1/2+2\ep$; hence, this function is positive in $(0,1+\ep]$. From~\eqref{eq-rho} and the comparison principle, it follows that $0<\rho(t,x)\leq \overline \rho(t,x)$ for all $t\ge T_\ep$ and $x\in\R$, where~$\overline \rho=\overline \rho (t,x)$ is the solution of the bistable Cauchy problem
\begin{equation*}
\overline \rho _t= d\,\overline \rho_{xx}+\overline{f}^{\ep}(\overline\rho)\hbox{ in }(T_\ep,+\infty)\times\R, \quad \overline \rho (T_\ep,\cdot)=\rho(T_\ep,\cdot)\hbox{ in }\R.
\end{equation*}
But since $\overline{f}^\ep$ is a bistable function in $[0,1+\ep]$ with negative integral over $[0,1+\ep]$, and since $\overline{\rho}(T_\ep,x)=\rho(T_\ep,x)\to0$ as $x\to\pm\infty$ by~\eqref{zero-infini}, together with $\sup_\R\overline{\rho}(T_\ep,\cdot)=\max_\R\rho(T_\ep,\cdot)<1+\ep$, one concludes from~\cite{Fife_McLeod_19} that~$\overline{\rho}(t,x)\to0$ as~$t\to+\infty$, uniformly in~$x\in\R$. Hence, $\rho(t,x)\to0$ as $t\to+\infty$, uniformly in $x\in\R$.

Now we consider the critical case
$$\tmin=\frac12.$$
In this case, the preceding argument does not work anymore and more care is needed. On the one hand, for any $(t,x)\in(0,+\infty)\times\R$ such that $\rho(t,x)\in[0,1]$, one has
$$\rho(t,x)\,(1-\rho(t,x))\,(\rho(t,x)-\tb(t,x))\leq\rho(t,x)\,(1-\rho(t,x))\,\Big(\rho(t,x)-\frac12\Big)$$
since $\tb(t,x)>\tmin=1/2$. On the other hand, for any $(t,x)\in(0,+\infty)\times\R$ such that~$\rho(t,x)>1$, one has (recall $\tb(t,x)<\tmax<1$)
$$\rho(t,x)(1-\rho(t,x))(\rho(t,x)-\tb(t,x))\leq -(1-\tb(t,x))(\rho(t,x)-1)\leq -(1-\tmax)(\rho(t,x)-1).$$
Therefore, for any $(t,x)\in(0,+\infty)\times\R$, there holds
$$\rho(t,x)\,(1-\rho(t,x))\,(\rho(t,x)-\tb(t,x))\leq\overline{f}(\rho(t,x)),$$
where 
\[
\overline{f}(s):=s(1-s)\Big(s-\frac12\Big)\mathbf{1}_{[0,1]}(s)-(1-\tmax)(s-1)\mathbf{1}_{(1,+\infty)}(s).
\]
By the comparison principle, one infers that, for any $T\geq 0$, $t\ge0$ and $x\in\R$, there holds
\begin{equation}\label{rhorhob}
\rho(t+T,x)\leq\rhob^T(t,x),
\end{equation}
where $\rhob^T$ denotes the solution of the Cauchy problem
\begin{equation}
\label{Cauchy_problem_theta_min_one_half}
\rhob^T_t=d\,\rhob^T_{xx}+\overline{f}(\rhob^T)\text{ in }(0,+\infty)\times\R,\ \ \ \rhob^T(0,\cdot)=\rho(T,\cdot)\text{ in }\R.
\end{equation}
This equation is a reaction--diffusion equation with a globally Lipschitz-continuous and piecewise smooth reaction term $\overline{f}$, whose derivative $\overline{f}'$ is well-defined except at $1$, where it only has well-defined left-sided and right-sided negative derivatives. For the underlying ordinary differential equation, the steady state $0$ is locally asymptotically stable from above, the steady state $1/2$ is unstable and, even though the flow is not $C^1$ at $1$, the steady state~$1$ is locally asymptotically stable from above and from below. Moreover,~$\int_0^1 \overline{f}=0$. We only need to find $T\geq 0$ such that $\rhob^T$ converges to $0$ as $t\to+\infty$ uniformly in $x\in\R$. Since the reaction term is not completely standard, we briefly recall how such a fact is proved.

Let $p:\R\to(0,1)$ be the standing wave solution of $-d\,p''=\overline{f}(p)$ in $\R$ with $p'<0$ in~$\R$ and limits $1$ and $0$ at $-\infty$ and $+\infty$ respectively. Since $\overline{f}(s)=s(1-s)(s-1/2)$ for~$s\in[0,1]$, the existence and uniqueness (up to shifts) of $p$ is standard. Let $s_0>0$, $r_0>0$ and $\gamma>0$ to be chosen later, and denote $s(t)=s_0\,\textup{e}^{-\gamma t}$ and $r(t)=r_0\,\textup{e}^{-\gamma t}$. The function $\overline{p}(t,x)=p(x+s(t))+r(t)$ defined for $(t,x)\in[0,+\infty)\times\R$ satisfies
\begin{align*}
\overline{p}_t(t,x)-d\,\overline{p}_{xx}(t,x)-\overline{f}(\overline{p}(t,x)) & = p'(x+s(t))s'(t)+r'(t)-d\,p''(x+s(t))-\overline{f}(\overline{p}(t,x)) \\
& = -\gamma\big(r_0+s_0\,p'(x+s(t))\big)\,\textup{e}^{-\gamma t}\\
&\qquad +\frac{\overline{f}(p(x+s(t)))-\overline{f}(p(x+s(t))+r(t))}{r(t)}r(t) \\
& = \textup{e}^{-\gamma t}\left[ -\gamma (r_0+s_0p'(x+s_0 z))-r_0\, G(p(x+s_0 z), r_0 z) \right]
\end{align*}
for all $t\ge0$ and $x\in\R$, where, on the last line above, $z=z(t)=\textup{e}^{-\gamma t}$ and one defines
$$G(q,y) = \frac{\overline{f}(q+y)-\overline{f}(q)}{y}$$
for $q\ge0$ and $y>0$. The function $G$ is extended at $y=0$ by $G(q,0)=\overline{f}'(q)$ if $q\in[0,1)\cup(1,+\infty)$ and by $G(1,0)=\lim_{y\to 0^+}G(1,y)=-(1-\tmax)$. The function $G$ is then continuous in~$[0,+\infty)^2\setminus\{(1,0)\}$ and it is locally bounded in $[0,+\infty)^2$. Now, we claim that we can choose the parameters $s_0, r_0, \gamma$ so that the function 
\[
H:(x,z)\in\R\times(0,1]\mapsto -\gamma (r_0+s_0p'(x+s_0 z)) -r_0 G(p(x+s_0 z), r_0 z)
\]
is nonnegative in $\R\times(0,1]$ (which will imply that $\overline{p}_t\ge d\,\overline{p}_{xx}+\overline{f}(\overline{p})$ in $[0,+\infty)\times\R$). To show the nonnegativity of $H$, let first select $\sigma\in(0,1/2)$ such that $\overline{f}'<0$ in $[0,\sigma]\cup[1-\sigma,1)\cup(1,+\infty)$, and denote 
\[
\Gamma_0=\min_{q\in[0,\sigma/2],\,y\in[0,\sigma/2]}(-G(q,y))>0,
\]
\[
\Gamma_1=\inf_{q\in[1-\sigma,1],\,y\in[0,\sigma]}(-G(q,y))>0,\footnote{One has $\Gamma_1>0$ since $\sup_{[1-\sigma,1)\cup(1,1+\sigma]}\overline{f}'<0$.}
\]
\[
\overline{\Gamma}=\sup_{q\in[0,1],\,y\in[0,\sigma]}|G(q,y)|>0,
\]
and
\[
\kappa = \min_{y\in[p^{-1}(1-\sigma),p^{-1}(\sigma/2)]}(-p'(y))>0.
\]
Let us then observe that:
\begin{itemize}
\item if $p(x+s_0 z)\leq\sigma/2$ and $r_0\leq\sigma/2$, then $H(x,z)\geq r_0(\Gamma_0-\gamma)$ (recall that $-\gamma s_0 p'\geq 0$);
\item if $p(x+s_0 z)\geq 1-\sigma$ and $r_0\leq\sigma$, then similarly $H(x,z)\geq r_0(\Gamma_1-\gamma)$;
\item if $p(x+s_0 z)\in[\sigma/2,1-\sigma]$ and $r_0\leq\sigma$, then $H(x,z)\geq \gamma s_0 \kappa -r_0(\gamma+\overline{\Gamma})$.
\end{itemize}
Taking for instance
$$r_0=\frac{\sigma}{2}>0,\ \ \gamma=\min(\Gamma_0,\Gamma_1)>0,\ \hbox{ and } \ s_0=\frac{r_0(\gamma+\overline{\Gamma})}{\gamma\kappa}>0.$$
the claim is proved, that is, $H\ge0$ in $\R\times(0,1]$. Hence,
$$\overline{p}_t\ge d\,\overline{p}_{xx}+\overline{f}(\overline{p})\ \hbox{ in }[0,+\infty)\times\R.$$
Notice that the same inequality holds by replacing the function $\overline{p}$ by the $x$-reflected one: $(t,x)\mapsto\overline{p}(t,-x)$. 

By virtue of \eqref{zero-infini} and \eqref{eq:limsup-mass}, together with the positivity of $\rho$ in $(0,+\infty)\times\R$, there exist $T^\star\geq 0$ and $x^\star\in\R$ such that $$0<\rho(T^\star,x)\leq \min\left(\overline{p}(0,x-x^\star),\overline{p}(0,-x-x^\star)\right)\ \hbox{ for all $x\in\R$}.$$
From now on, $\rhob=\rhob^{T^\star}$ is the solution of~\eqref{Cauchy_problem_theta_min_one_half} with $T=T^\star$. Due to the preceding calculations and by the comparison principle, the inequality
$$0<\rhob(t,x)\leq \min\left(\overline{p}(t,x-x^\star),\overline{p}(t,-x-x^\star)\right)$$
holds for all $t\geq 0$ and $x\in\R$. Therefore there exists $T_1>0$ such that $\sup_{x\in\R}\rhob(T_1,x)<1$. Now, using standard heat kernel estimates \cite{Fri-64}, we deduce that $\rhob(T_1,x)=O(\textup{e}^{-Cx^2})$ as~$x\to\pm\infty$ for some constant $C>0$. Since $p(x)\sim C'\textup{e}^{-x/\sqrt{2d}}$ as $x\to+\infty$, for some constant $C'>0$, there exists $x_1\in\R$ such that
$$0<\rhob(T_1,x)\leq \min(p(x-x_1),p(-x-x_1))<1\ \hbox{ for all }x\in\R.$$ 
Since $\min(p(x-x_1),p(-x-x_1))$ is a super-solution of the stationary elliptic equation~$-d\,q''=\overline{f}(q)$, the solution $\hat{\rho}$ of
\begin{equation*}
    \begin{cases}
    \hat{\rho}_t =d\,\hat{\rho}_{xx}+\overline{f}(\hat{\rho}) & \text{in }(T_1,+\infty)\times\R,\\
    \hat{\rho}(T_1,x)=\min(p(x-x_1),p(-x-x_1)) & \text{for all }x\in\R,
    \end{cases}
\end{equation*}
is nonincreasing in time (and even decreasing) in $[T_1,+\infty)\times\R$. Hence $\hat{\rho}(t,\cdot)$ converges as~$t\to+\infty$ in $\mathcal{C}^2_\textup{loc}(\R)$, by standard parabolic estimates, to a $C^2(\R)$ solution $q:\R\to[0,1]$ of $-d\,q''=\overline{f}(q)$ in $\R$ with limit $0$ at $\pm\infty$. Since $\overline{f}(s)=s(1-s)(s-1/2)$ in $[0,1]$ and since it is well-known that the unique nonnegative solution of the equation $-d\,q''=q(1-q)(q-1/2)$ with limit~$0$ at $\pm\infty$ is $q\equiv0$, the long-time limit of $\hat{\rho}$ is identically~$0$. Using again the stationary super-solution $\min(p(x-x_1),p(-x-x_1))$ and the locally uniform convergence, it turns out that the convergence of $\hat{\rho}(t,\cdot)$ to $0$ is uniform in~$\R$ as~$t\to+\infty$.

Finally, by the comparison principle, $0<\rhob\leq\hat{\rho}$ in $[T_1,+\infty)\times\R$. Therefore $\rhob$ converges uniformly in space to $0$ as $t\to+\infty$, and then so does $\rho$ by~\eqref{rhorhob}. The proof of Theorem~\ref{thm:syst-ext} is thereby complete.
\end{proof}

\subsection{$\lambda_\alpha\le0$ makes extinction impossible}\label{sec:imposs_extin}

As claimed at the beginning of Section~\ref{s:ext-pers}, the sign of the principal eigenvalue $\lambda_\alpha$ of~\eqref{vp} decides between systematic persistence and possible extinctions. Notice that, in typical Fisher-KPP situations, all solutions go extinct as soon as $\lambda_\alpha \geq 0$. Our results on equation~\eqref{eq} are in sharp contrast: first the critical case $\lambda_\alpha=0$ implies persistence (in some sense, it corresponds to a degenerate monostable situation for which the hair trigger effect \cite{AW-78} does hold); next the case $\lambda_\alpha >0$ leads to both possible extinction and possible persistence, see Theorems~\ref{th:ext2} and~\ref{thm:poss_persis_subcrit} in Sections~\ref{sec:poss_extin} and~\ref{sec:poss_persis} below (in some sense, the situation is similar to that of a local bistable reaction-diffusion equation). The reason is that, as explained in Section~\ref{sec1}, the underlying nature of model~\eqref{eq} may vary from Fisher-KPP to bistable, not to mention degenerate monostable.

We deal in the present Section~\ref{sec:imposs_extin} with the systematic persistence when $\lambda_\alpha\le0$ (lines~2 and~3 of Table~\ref{table:summary}).

\begin{thm}[Systematic persistence when $\lambda_\alpha\le0$]\label{th:rescue}
Let $\lambda_\alpha$ be the principal eigenvalue of the eigenproblem~\eqref{vp}--\eqref{rayleigh}. If $\lambda_\alpha\leq 0$, then every solution~$u$ of the Cauchy problem~\eqref{eq}--\eqref{initial} persists. Furthermore, if $\lambda_\alpha<0$, then
\begin{equation}\label{liminf}
\inf_{(x,\theta)\in\R\times\overline{\Theta}}\Big(\liminf_{t\to+\infty}u(t,x,\theta)\Big)>0\ \hbox{ and }\ \inf_{x\in\R}\Big(\liminf_{t\to+\infty}\rho(t,x)\Big)>0.
\end{equation}
Lastly, if $\tmax\le0$, then $\lambda_\alpha<0$ and
\begin{equation}\label{limits2}
\left\{\begin{array}{llll}
\rho(t,\cdot) & \!\!\!\longrightarrow\! & 1 & \hbox{as $t\to+\infty$ locally uniformly in $\R$},\vspace{3pt}\\
u(t,\cdot,\cdot) & \!\!\!\longrightarrow\! & \displaystyle\frac{1}{\tmax-\tmin} & \hbox{as $t\to+\infty$ locally uniformly in $\R\times\overline{\Theta}$}.\end{array}\right.
\end{equation}
\end{thm}

\begin{proof} {\it Step 1: persistence in the case $\lambda_\alpha<0$}. Let us first assume that $\lambda_\alpha<0$, and consider a solution~$u$ of the Cauchy problem~\eqref{eq}--\eqref{initial}. From~\eqref{eq:max_rho} and~\eqref{ineg0} (with here~$T^*=+\infty$, any~$p>1$ and, say, $\tau=1$) the nonlinear term in \eqref{eq} satisfies, for times~$t\geq 1$,
\begin{equation}
\label{estimate-reaction-term}
u(\rho -\theta)(1-\rho)=-\theta u  +\theta \rho u +u\rho(1-\rho) \geq -\theta u -K u^{1+1/p},
\end{equation}
for some positive constant $K$ depending on $\tmin$, $\tmax$, $C_{p,1}$ and $M$. As a result, we can compare the nonlocal problem~\eqref{eq} with a local problem: namely,  $u=u(t,x,\theta)$ satisfies
\begin{equation}
\label{u:sur-sol}
\mathcal L u:=u_t-d u_{xx}-\alpha u_{\theta \theta}+\theta u+K u^{1+1/p}\geq 0\ \hbox{ for all }t\geq 1\hbox{ and }(x,\theta)\in\R\times\overline{\Theta}.
\end{equation}
Remember also that $u_\theta(t,x,\tmin)= u_\theta(t,x,\tmax)=0$ for all $t\geq 1$ and $x\in \R$, and that~$u(1,x,\theta)>0$ for all $x\in \R$ and $\theta\in\overline{\Theta}$ from Section~\ref{sec21}. Now, for $R>0$ and $\ep>0$, consider the compactly supported continuous function $w$ defined in $[-R,R]\times\overline{\Theta}$ by
\begin{equation}
\label{subsol-compactly}
w(x,\theta):=\ep \sin\Big(\frac{\pi}{2R}(x+R)\Big)\,\varphi_\alpha(\theta),
\end{equation}
where $\varphi_\alpha$ is the unique solution of~\eqref{vp} such that $\max_{\overline{\Theta}}\varphi_\alpha=1$, with $\lambda=\lambda_\alpha$ given in~\eqref{rayleigh}. A straightforward computation shows that $\mathcal L w\leq 0$ in $[-R,R]\times\overline{\Theta}$ as soon as 
\begin{equation}\label{goal}
\frac{d\pi^2}{4R^2}+\lambda_\alpha+K\ep^{1/p}\leq 0.
\end{equation}
Due to the negativity of $\lambda_\alpha$, the above inequality is true by selecting $R=R_0>0$ large enough (so that $d\pi^2/(4R_0^2)+\lambda_\alpha<0$) and then $\ep=\ep_0>0$ small enough. Moreover, up to reducing~$\ep_0>0$, we have $w(x,\theta)\leq u(1,x,\theta)$ in $[-R_0,R_0]\times\overline{\Theta}$. Observe also that~$w(\pm R_0,\theta)=0<u(t,\pm R_0,\theta)$ for all $t\ge1$ and $\theta\in\overline{\Theta}$, and that $w_\theta(x,\tmin)=w_\theta(x,\tmax)=0$ for all $x\in[-R_0,R_0]$. We therefore deduce from the comparison principle that 
$$w(x,\theta)\le u(t,x,\theta)\ \hbox{ for all }t\ge1\hbox{ and }(x,\theta)\in[-R_0,R_0]\times\overline{\Theta},$$
and thus $u$ cannot go extinct, that is, persistence necessarily occurs.

Les us  now show the stronger property~\eqref{liminf}. Consider any sequence $(t_n)_{n\in\mathbb{N}}$ of positive real numbers diverging to $+\infty$. From standard parabolic estimates, the functions
$$u_n:(t,x,\theta)\mapsto u_n(t,x,\theta):=u(t+t_n,x,\theta)\ \hbox{ and }\ \rho_n:(t,x)\mapsto\rho_n(t,x):=\rho(t+t_n,x)$$
converge up to extraction of a subsequence, in $\mathcal{C}^{1;2}_{t;(x,\theta);\textup{loc}}(\R\times\R\times\overline{\Theta})$ and in $\mathcal{C}^{1;2}_{t;x;\textup{loc}}(\R\times\R)$ respectively, to some nonnegative bounded functions $u_\infty\in\mathcal{C}^{1;2}_{t;(x,\theta)}(\R\times\R\times\overline{\Theta})$ and $\rho_\infty=\int_\Theta u_\infty(\cdot,\cdot,\theta)\,\textup{d}\theta\,\in\mathcal{C}^{1;2}_{t;x}(\R\times\R)$ solving~\eqref{eq}--\eqref{neumann} with $t\in\R$. Furthermore, with the same choice of parameters $(\ep_0,R_0)$ as in the previous paragraph, one has
\begin{equation}\label{inf1}
\inf_{t\in\R}\Big(\min_{\overline{\Theta}}u_\infty(t,0,\cdot)\Big)\ge\min_{\overline{\Theta}}w(0,\cdot)=\ep_0\min_{\overline{\Theta}}\varphi_\alpha>0.
\end{equation}
Since the function $(t,x,\theta)\mapsto(\rho_\infty(t,x)-\theta)(1-\rho_\infty(t,x))$ is globally bounded in $\R\times\R\times\overline{\Theta}$, the Harnack inequality yields, for each compact set $\mathcal{K}\subset\R\times\overline{\Theta}$, the existence of a constant~$\mu>0$ such that $u_\infty(t+1,x,\theta)\ge\mu\,u_\infty(t,0,(\tmin+\tmax)/2)$ for all~$t\in\R$ and~$(x,\theta)\in\mathcal{K}$. Hence, together with~\eqref{inf1}, one gets that
\begin{equation}\label{infK}
\inf_{t\in\R}\Big(\min_{(x,\theta)\in\mathcal{K}}u_\infty(t,x,\theta)\Big)>0
\end{equation}
for each compact set $\mathcal{K}\subset\R\times\overline{\Theta}$. Consider now any $x_0\in\R$, and define
$$w^{x_0}(x,\theta)=\sin\Big(\frac{\pi}{2R_0}\,(x-x_0+R_0)\Big)\,\varphi_\alpha(\theta)$$
for all $(x,\theta)\in[x_0-R_0,x_0+R_0]\times\overline{\Theta}$, with $R_0>0$ (and $\ep_0>0$) satisfying~\eqref{goal}. From~\eqref{infK} with $\mathcal{K}=[x_0-R_0,x_0+R_0]\times\overline{\Theta}$, the quantity
$$\ep^*=\sup\big\{\ep\in[0,\ep_0]:\ep\,w^{x_0}\le u_\infty\hbox{ in }\R\times[x_0-R_0,x_0+R_0]\times\overline{\Theta}\big\}$$
is a positive real number, that is $0<\ep^*\le\ep_0$. We claim that
$$\ep^*=\ep_0.$$
Assume not. Then $\ep^*<\ep_0$ and, using~\eqref{infK} again, there exist a point
$$(x^*,\theta^*)\in(x_0-R_0,x_0+R_0)\times\overline{\Theta}$$ 
and a sequence $(t'_n)_{n\in\mathbb{N}}$ in $\R$ such that the functions
$$(t,x,\theta)\mapsto u_\infty(t+t'_n,x,\theta)\ \hbox{ and }\ (t,x)\mapsto\rho_\infty(t+t'_n,x)$$
converge in $\mathcal{C}^{1;2}_{t;(x,\theta);\textup{loc}}(\R\times\R\times\overline{\Theta})$ and in $\mathcal{C}^{1;2}_{t;x;\textup{loc}}(\R\times\R)$ respectively, to some nonnegative bounded functions $U_\infty\in\mathcal{C}^{1;2}_{t;(x,\theta)}(\R\times\R\times\overline{\Theta})$ and $\varrho_\infty=\int_\Theta U_\infty(\cdot,\cdot,\theta)\,\textup{d}\theta\,\in\mathcal{C}^{1;2}_{t;x}(\R\times\R)$ solving~\eqref{eq}--\eqref{neumann} with $t\in\R$, and such that $\ep^*w^{x_0}\le U_\infty$ in $\R\times[x_0-R_0,x_0+R_0]\times\overline{\Theta}$ with equality at $(0,x^*,\theta^*)$. But since $U_\infty$ satisfies~\eqref{u:sur-sol} in $\R\times\R\times\overline{\Theta}$ (it is a super-solution), whereas $\ep^*w^{x_0}$ is a (stationary) sub-solution in $\R\times[x_0-R_0,x_0+R_0]\times\overline{\Theta}$ (from the choice $R_0$ and $\ep_0$), the strong parabolic maximum principle and the Hopf lemma imply that $\ep^*w^{x_0}\equiv U_\infty$ in $(-\infty,0]\times[x_0-R_0,x_0+R_0]\times\overline{\Theta}$, which is impossible on $(-\infty,0]\times\{x_0\pm R_0\}\times\overline{\Theta}$. Therefore, $\ep^*=\ep_0$, and $\ep_0w^{x_0}\le u_\infty$ in $\R\times[x_0-R_0,x_0+R_0]\times\overline{\Theta}$. In particular, one infers that $u_\infty(t,x_0,\theta)\ge\ep_0\min_{\overline{\Theta}}\varphi_\alpha$ for all $t\in\R$ and $\theta\in\overline{\Theta}$. Since $x_0$ was arbitrary in $\R$, one concludes that
$$\inf_{\R\times\R\times\overline{\Theta}}u_\infty\ge\ep_0\min_{\overline\Theta}\varphi_\alpha>0.$$
Since the sequence $(t_n)_{n\in\mathbb{N}}$ diverging to $+\infty$ was arbitrary, one finally gets~\eqref{liminf} for the function $u$, and then also for the mass $\rho$ by definition of $\rho$ and the local uniform convergence of the functions $u_n$ as $n\to+\infty$.

To complete this step 1, let us show the long-time behavior~\eqref{limits2} in the case $\tmax\le0$. First of all, Lemma~\ref{lem:vp} implies that $\lambda_\alpha<0$ in this case. With the same notations~$(t_n)_{n\in\mathbb{N}}$ and $(u_\infty,\rho_\infty)$ as in the previous paragraph, there is $\eta>0$ such that $\rho_\infty\ge\eta$ in~$\R\times\R$ (without loss of generality, one can assume that $0<\eta<1$). Furthermore, $\rho_\infty\le1$ in~$\R\times\R$ by~\eqref{eq:limsup-mass}. On the other hand, as in~\eqref{eq-rho} in Section~\ref{sec21}, the function $\rho_\infty$ obeys
$$(\rho_\infty)_t=d\,(\rho_\infty)_{xx}+\rho_\infty\,(\rho_\infty-\tb_\infty)\,(1-\rho_\infty)\ \hbox{ in }\R\times\R,$$
with
$$\tb_\infty(t,x)=\frac{1}{\rho_\infty(t,x)}\int_\Theta\theta\,u_\infty(t,x,\theta)\,\textup{d}\theta\,\in(\tmin,\tmax)$$
for all $(t,x)\in\R\times\R$. Since $\rho_\infty(1-\rho_\infty)\ge0$ in $\R\times\R$ and $\tmax\le0$, one gets that $-\rho_\infty\,\tb_\infty\,(1-\rho_\infty)\ge0$ in $\R\times\R$, hence
$$(\rho_\infty)_t\ge d\,(\rho_\infty)_{xx}+\rho_\infty^2(1-\rho_\infty)\ \hbox{ in $\R\times\R$}.$$
Let $\zeta:\R\to\R$ be the solution of $\zeta'(s)=\zeta(s)^2(1-\zeta(s))$ with $\zeta(0)=\eta\in(0,1)$. Notice that~$\zeta(s)\to1$ as $s\to+\infty$. The maximum principle implies that, for any real numbers~$t_0<t$, one has $\rho_\infty(t,\cdot)\ge\zeta(t-t_0)$ in $\R$. The passage to the limit as $t_0\to-\infty$ implies that $\rho_\infty(t,\cdot)\ge1$ for every $t\in\R$, and finally $\rho_\infty\equiv1$ in $\R\times\R$. From the equation~\eqref{eq} satisfied by the pair $(u_\infty,\rho_\infty)$ with $t\in\R$, one then gets that the bounded function $u_\infty$ satisfies the linear heat-like equation $(u_\infty)_t=d\,(u_\infty)_{xx}+\alpha\,(u_\infty)_{\theta\theta}$ in $\R\times\R\times\overline{\Theta}$ with Neumann boundary conditions on $\R\times\R\times\partial\Theta$. Therefore, it is standard to conclude that $u_\infty$ is then constant in $\R\times\R\times\overline{\Theta}$.\footnote{To get this property, notice first that, from standard parabolic estimates, the function $u_\infty$ is of class~$\mathcal{C}^\infty(\R\times\R\times\overline{\Theta})$ with bounded derivatives at any order. The function $(u_\infty)_\theta$ satisfies the same equation as~$u_\infty$, but with homogeneous Dirichlet boundary condition on $\R\times\R\times\partial\Theta$. If $M_\infty:=\sup_{\R\times\R\times\overline{\Theta}}(u_\infty)_\theta>0$, then there is a sequence $(t_n,x_n,\theta_n)_{n\in\N}$ in $\R\times\R\times\overline{\Theta}$ such that the functions $(u_\infty)_\theta(\cdot+t_n,\cdot+x_n,\cdot)$ converge in $\mathcal{C}^{1;2}_{t;(x,\theta);loc}(\R\times\R\times\overline{\Theta})$ to a bounded solution $v$ of the same equation, with Dirichlet boundary condition on~$\R\times\R\times\partial\Theta$, and $v(0,0,\theta_\infty)=M_\infty=\sup_{\R\times\R\times\overline{\Theta}}v>0$ for some $\theta_\infty\in\overline{\Theta}$. This contradicts the strong parabolic maximum principle and Hopf lemma. Therefore, $(u_\infty)_\theta\le0$ in $\R\times\R\times\overline{\Theta}$, and similarly $(u_\infty)_\theta\ge0$ in $\R\times\R\times\overline{\Theta}$. Finally, the function $u_\infty$ does not depend on $\theta$ and it is a bounded entire solution of the heat equation $(u_\infty)_t=d(u_\infty)_{xx}$ in $\R\times\R$. It is then well known that it must be constant.} Since $\rho_\infty\equiv1$ in $\R\times\R$, one then infers that~$u_\infty\equiv1/(\tmax-\tmin)$ in $\R\times\R\times\overline{\Theta}$. Finally, the limits $(u_\infty,\rho_\infty)$ do not depend on the original sequence $(t_n)_{n\in\N}$ nor on any subsequence, and~\eqref{limits2} follows.

\medskip
\noindent{\it Step 2: persistence in the case $\lambda_\alpha=0$}. By Lemma~\ref{lem:vp}, one has $0=\lambda_\alpha<(\tmin+\tmax)/2$, so that necessarily $\tmin>-\tmax>-1$. Then we define
$$\nu=\frac{\tmin+1}{2}>0$$
and, without loss of generality, we assume the existence of a real number~$t_0\geq 1$ such that $\sup_{\R}\rho(t_0,\cdot)<\nu$ (otherwise, we would have $\liminf_{t\to+\infty}\sup_\R\rho(t,\cdot)\ge\nu>0$ and we would have obtained the desired result). Define
\begin{equation*}
T=\sup\Big\{ t\geq t_0 :\forall \tau\in[t_0,t),\ \sup_{\R}\rho(\tau,\cdot)\leq \nu \Big\}.
\end{equation*}
By continuity of $\rho$ with respect to $t$ in the sense of the uniform topology in $x\in\R$ (by Proposition~\ref{prop:well-posedness}), one knows that $T>t_0$. Let us then prove that $T<+\infty$, which will end the proof, from the arbitrariness of $t_0\ge1$ with $\sup_\R\rho(t_0,\cdot)<\nu$.

To show that $T<+\infty$, notice first that
$$u_t-d u_{xx}-\alpha u_{\theta\theta} +\theta u= \theta\rho u+\rho u(1-\rho)= (\theta+1-\nu)\rho u + \rho u(\nu-\rho)$$
in $(0,+\infty)\times\R\times\overline{\Theta}$. Since $\theta+1-\nu=(1+\theta)/2+(\theta-\tmin)/2\geq \nu>0$, we deduce directly from~\eqref{ineg0} (with, say, $\tau=1$ and any $p>1$) and the positivity of $u$ and $\rho$, that
$$u_t-du_{xx}-\alpha u_{\theta\theta}+\theta u=(\theta+1-\nu)\rho u+\rho u(\nu-\rho)\geq \frac{\nu}{C_{p,1}}\,u^{p+1}$$
for all $t\in[t_0,T)$ and $(x,\theta)\in\R\times\overline{\Theta}$. Let now $\ep>0$ and let $v=v(t,x)$ be the solution of 
\begin{equation*}
    \begin{cases}
    \displaystyle v_t-d v_{xx}=\frac{\nu}{C_{p,1}}\,v^{p+1} & \text{in }(t_0,t_1)\times\R, \\
    \displaystyle v(t_0,x)=\ep\times\Big(\min_{\overline{\Theta}}u(t_0,x,\cdot)\Big) & \text{for }x\in\R,
    \end{cases}
\end{equation*}
with maximal existence time-interval $[t_0,t_1)$ with $t_0<t_1\le+\infty$. From Section~\ref{sec21}, one knows that $v(t_0,x)>0$ for each $x\in\R$. Then, define
$$\underline{u}(t,x,\theta)=\ep^{-1} v(t,x)\,\varphi_\alpha(\theta)$$
for $(t,x,\theta)\in[t_0,t_1)\times\R\times\overline{\Theta}$. By construction,
\begin{equation*}
    \underline{u}_t-d \underline{u}_{xx}-\alpha \underline{u}_{\theta\theta} +\theta \underline{u}(t,x,\theta)-\frac{\nu}{C_{p,1}}\,\underline{u}(t,x,\theta)^{p+1} =\frac{\nu}{\ep C_{p,1}}\,\varphi_{\alpha}(\theta)\,v(t,x)^{p+1}\left[ 1-\ep^{-p}\varphi_\alpha(\theta)^p \right]
\end{equation*}
in $(t_0,t_1)\times\R\times\overline{\Theta}$. Up to decreasing the value of $\ep$, we can assume that $1-\ep^{-p}\varphi_\alpha(\theta)^p\leq 0$ for all $\theta\in\overline{\Theta}$ (more precisely, it suffices to assume that $0<\ep\leq \min_{\overline{\Theta}}\varphi_\alpha$), so that the right-hand side above is nonpositive. Moreover, 
$$\underline{u}(t_0,x,\theta)=\ep^{-1} v(t_0,x)\,\varphi_\alpha(\theta)=\Big(\min_{\overline{\Theta}}u(t_0,x,\cdot)\Big)\times\varphi_\alpha(\theta)\leq \min_{\overline{\Theta}}u(t_0,x,\cdot)\leq u(t_0,x,\theta)$$
for all $(x,\theta)\in\R\times\overline{\Theta}$. In the end, the nonnegative functions $\underline{u}$ and $u$ are respectively a subsolution and a supersolution of the same local reaction--diffusion equation in~$(t_0,\min(t_1,T))\times\R\times\overline{\Theta}$ with ordered values at time~$t_0$, so that $0\le\underline{u}(t,\cdot,\cdot)\leq u(t,\cdot,\cdot)$ in~$\R\times\overline{\Theta}$ for all times $t\in[t_0,\min(t_1,T))$. From the seminal blow-up result of Fujita \cite{Fuj-66} (the critical case being later completed by \cite{Hay-73} and \cite{Kob-Sir-Tun-77}, see also \cite{AW-78}), $v$ blows up as soon as $p+1\leq 3$, that is, $\|v(t,\cdot)\|_{L^\infty(\R)}\to+\infty$ as $t\to t_1$. Therefore, picking any $p\in(1,2]$, we deduce from the global boundedness of $u$ that $T<+\infty$.

As already noticed, this shows the persistence of $\rho$, and then that of $u$. The proof of Theorem~\ref{th:rescue} is thereby complete.
\end{proof}

\begin{rem}
As a consequence of Lemma~\ref{lem:vp} and Theorem~\ref{th:rescue}, some typical situations leading to evolutionary rescue (systematic persistence, even for small initial data) are the following:
\begin{itemize}
\item when the phenotypic space \lq\lq leans to the left'', that is  $\tmin+\tmax\leq 0$, and this whatever the mutation coefficient $\alpha>0$;
\item when $\tmin<0$ and the mutation coefficient $\alpha>0$ is small enough compared with $-\tmin$, a sufficient condition being
$$\frac 32 \left(\frac{\pi ^2\alpha}2\right)^{1/3}+\tmin<0,$$
and this whatever the maximal phenotypic trait $\tmax$.
\end{itemize}
In the above two cases, one has $\lambda_\alpha<0$ and the solutions of~\eqref{eq}--\eqref{initial} escape from a uniform-in-$(x,\theta)$ neighborhood of $0$ at large times, in the sense of~\eqref{liminf}.
\end{rem}

\begin{rem}
From Lemma~\ref{lem:vp}, the condition $\lambda_\alpha\le0$ implies that $\tmin<0$. If one further assumes that $\tmax\le0$, then $\lambda_\alpha$ is necessarily negative and in that case the population density has a well-characterized limit at long time, by Theorem~\ref{th:rescue}. For sign-changing traits ($\tmin<0<\tmax$) with $\tmin+\tmax>0$, the situation is more complex and $\lambda_\alpha$ may be nonpositive or positive, according to the value of $\alpha$. Small populations may manage to stay mainly in the zone of  negative traits, where they have a chance to escape extinction. A consequence of Theorem~\ref{th:rescue} above and Theorem~\ref{th:ext2} below is that what decides whether this evolutionary rescue happens or not is the sign of $\lambda_\alpha$.
\end{rem}

\subsection{$\lambda_\alpha>0$ makes extinction possible}\label{sec:poss_extin} 

In contrast with the previous section which was concerned with the case $\lambda_\alpha\le0$ and the systematic persistence, we consider in this section the case $\lambda_\alpha>0$ and we show the possibility of extinction in this case, that is, persistence is not systematic (this corresponds to the intersection of line~4 and columns~3-7 of Table~\ref{table:summary}).

From Lemma~\ref{lem:vp}, the condition $\lambda_\alpha>0$ implies that $\tmin+\tmax>0$, hence $\tmax>0$. The traits may then be nonnegative, or sign-changing. We first consider in Proposition~\ref{prop:ext} below the case of nonnegative traits, that is, $\tmin\ge0$. In this case, there is no ``refuge'' for small enough populations and therefore it is natural to guess that these populations go extinct: in other words, extinction is possible. Actually, the condition~$\tmin\ge0$ yields $\lambda_\alpha>0$ by Lemma~\ref{lem:vp}, and Proposition~\ref{prop:ext} can then be viewed as a particular case of the following Theorem~\ref{th:ext2}, which deals with the more general case~$\lambda_\alpha>0$. But we chose to first consider separately the case $\tmin\ge0$ in Proposition~\ref{prop:ext} since it is easier to deal with, and since the proof involves some different arguments as those of Theorem~\ref{th:ext2} below.

\begin{prop}[Possible extinctions when $\tmin\ge0$]\label{prop:ext} If $\tmin\geq 0$, then sufficiently small initial data of the Cauchy problem \eqref{eq}--\eqref{initial} lead to extinction.
\end{prop}

\begin{proof} On the one hand, by~\eqref{eq:max_rho}, one has $0\le \rho \le \max(M,1)$ in $[0,+\infty)\times\R$. On the other hand, $\tmin<\tb(t,x)<\tmax<1$ by~\eqref{tminmax}, and
$$(\rho(t,x)-\tb(t,x))\,(1-\rho(t,x))\le\frac{(1-\tb(t,x))^2}{4}\le\frac{(1-\tmin)^2}{4}=:C$$
for all $t>0$ and $x\in \R$. By comparison with a linear ordinary differential equation, we get
\begin{equation}
    \label{eq:rho(1,.)}
    0< \rho(1,\cdot)\le  M\textup{e}^{C}\ \hbox{ in }\R.
\end{equation}
Since $\Theta \subset [0,+\infty)$ here by assumption, we can choose $\tau=1$ and any $p>1$ in~\eqref{ineg} and integrate it against $\theta$ over $\Theta$ to reach
\begin{equation*}
K_1 \rho^{p-1}(t,x)\leq \overline \theta (t,x)\leq K_2 \rho^{1/p\,-1}(t,x)\ \hbox{ for all }t\ge1\hbox{ and }x\in\R,
\end{equation*}
where $K_i=K_i(p)>0$. By restricting to $1<p<2$, assuming $M\leq 1$ and using the above estimate, we deduce from~\eqref{eq-rho} the inequality
$$\rho_t-d\rho_{xx} \leq \rho ^p (\rho^{2-p}-K_1)\,(1-\rho)\ \text{ in }[1,+\infty)\times\R.$$
By direct comparison with the underlying ordinary differential equation, starting from $\sup_\R\rho(1,\cdot)$ and using~\eqref{eq:rho(1,.)}, we get that $\sup _{x\in \R} \rho(t,x)\to 0$ as $t\to +\infty$ as soon as $M$ also satisfies
$$M\textup{e}^{C}<\min\left(K_1^{1/(2-p)},1\right),$$
which concludes the proof.
\end{proof}

\begin{rem}
It might be tempting to use the same technique to get a persistence result for large initial data. However,
in order to obtain an inequality of the type 
\begin{equation*}
    \rho_t-d\rho_{xx} \geq \rho ^{1/p} (\rho^{2-1/p}-K_2)(1-\rho),
\end{equation*}
we would need to multiply $\rho-\overline \theta \geq \rho-K_2\rho^{1/p\,-1}$ by $\rho(1-\rho)\geq 0$. This requires $M\leq 1$, and then we could think of a persistence result by comparison and standard results on local bistable reaction--diffusion equations only if $0<K_2<1$ is small enough so that $\int_0^1s^{1/p}(s^{2-1/p}-K_2)(1-s)\,\textup{d}s>0$, which is typically false as $K_2$ is a large constant.
\end{rem}

\begin{thm}[Possible extinction when $\lambda_\alpha>0$]\label{th:ext2}
Let $\lambda_\alpha$ be the principal eigenvalue of the eigenproblem~\eqref{vp}--\eqref{rayleigh}. If $\lambda_\alpha>0$, then sufficiently small initial data of the Cauchy problem~\eqref{eq}--\eqref{initial} lead to extinction. Furthermore, the condition 
\begin{equation}\label{u0extinction}
\|u_0\|_{L^\infty(\R\times\Theta)}<\frac{\lambda_\alpha\,\min_{\overline{\Theta}}\varphi_\alpha}{(\tmax-\tmin)\,(1+\tmax)}
\end{equation}
is a sufficient condition for extinction.
\end{thm}

\begin{proof}
Let us now assume $\lambda_\alpha>0$. Fix $\mu\in(0,\lambda_\alpha)$, and then any $\ep$ such that
\begin{equation}\label{choiceeps}
0<\ep<\frac{\lambda_\alpha-\mu}{(\tmax-\tmin)(1+\tmax)}.
\end{equation}
Defining
$$w(t,x,\theta)=w(t,\theta):=\ep\,\textup{e}^{-\mu t}\varphi_\alpha(\theta),$$
where $\varphi_\alpha$ is as in~\eqref{vp}--\eqref{rayleigh} with $\lambda=\lambda_\alpha$, we find that
\begin{equation}\label{compute-H}
\mathcal H w:=w_t-dw_{xx}-\alpha w_{\theta \theta}+\theta w=\ep\, \textup{e}^{-\mu t}\varphi_\alpha(\theta)\left(\lambda_\alpha-\mu\right)=(\lambda_\alpha-\mu)w>0
\end{equation}
in $\R\times\R\times\overline{\Theta}$. Pick a nontrivial and nonnegative initial condition $u_0\in\mathcal{C}_c(\R\times\overline{\Theta})$ such that, for every $\theta\in\overline{\Theta}$, $\max_{\R} u_0(\cdot,\theta)< \ep\,\varphi_\alpha(\theta)$ (it is therefore sufficient to take $u_0\in\mathcal{C}_c(\R\times\overline{\Theta},[0,+\infty))$ such that $\max_{\R\times\overline{\Theta}}u_0=\|u_0\|_{L^\infty(\R\times\Theta)}<\ep\,\min_{\overline{\Theta}}\varphi_\alpha$). Notice that, by continuity with respect to~$\theta\in\overline{\Theta}$, there is $\eta>0$ such that $\max_{\R} u_0(\cdot,\theta)< \ep\,\varphi_\alpha(\theta)-\eta$ for all $\theta\in\overline{\Theta}$. We then define
$$T:=\sup\left\{\tau \geq 0:\forall\,t\in[0,\tau),\ \forall\,\theta\in\overline{\Theta},\ \sup_{\R} u(t,\cdot,\theta)\leq w(t,\theta)\right\}.$$
Notice that, due to the well-posedness (Proposition~\ref{prop:well-posedness}) and more specifically to the continuity of $u$ when $t\to 0^+$ in $L^\infty(\R\times\Theta)$, one has $T>0$.

We are going to show that $T=+\infty$ and, to do so, we assume by way of contradiction that $T<+\infty$. From the behavior~\eqref{zero-infini} at large $|x|$ and again from the continuity of $u$ with respect to $t$ in the sense of $L^\infty(\R\times\Theta)$, there must be a touching point~$(x_0,\theta_0)\in \R\times \overline \Theta$ such that the function $\psi:=w-u$ satisfies $\psi \geq 0$ in~$[0,T]\times \R \times \overline \Theta$ and~$\psi(T,x_0,\theta_0)=0$. In particular, there holds $\psi_t(T,x_0,\theta_0)\leq 0$ and~$\psi_{xx}(T,x_0,\theta_0)\geq 0$. Moreover, since~$\psi$ satisfies the no-flux boundary condition on~$\partial \Theta$, we also have $\psi_{\theta\theta}(T,x_0,\theta_0)\geq 0$ whether $\theta_0$ be in~$\Theta$ or on~$\partial \Theta$. Then, from~\eqref{eq} and~\eqref{compute-H} together with the $\mathcal{C}^{1;2}_{t;(x,\theta)}((0,+\infty)\times\R\times\overline{\Theta})$ regularity of $u$, the function $\psi$ satisfies 
\begin{equation*}
    \psi_t -d\psi_{xx}-\alpha\psi_{\theta\theta}=w\left(\lambda_{\alpha}-\mu-\theta-\left(1-\frac{\psi}{w}\right)(\rho-\theta)(1-\rho)\right)
\end{equation*}
in $(0,+\infty)\times\R\times\overline{\Theta}$. Evaluating at $(T,x_0,\theta_0)$, we obtain
$$\lambda_\alpha- \mu \leq \rho(T,x_0)\,(1-\rho(T,x_0)+\theta_0)\leq \rho(T,x_0)\,(1+\tmax).$$
Since
$$\rho(T,x_0)=\int_\Theta u(T,x_0,\cdot)\le\int_\Theta w(T,\cdot)=\ep\,\textup{e}^{-\mu T}\int _{\Theta}\varphi_\alpha(\theta)\leq \ep\,(\tmax-\tmin),$$
and since the assumption $\lambda_\alpha>0$ yields $\tmin+\tmax>0$ by Lemma~\ref{lem:vp} (and then~$\tmax>0$ and~$1+\tmax>0$), we end up with
$$\lambda_\alpha- \mu \leq \ep\,(\tmax-\tmin)\,(1+\tmax),$$
which is a contradiction from the above choice of $\ep$. As a result, for each $u_0$ small enough so that
$$\max_\R\,u_0(\cdot,\theta)<\ep\,\varphi_\alpha(\theta)$$
for all $\theta\in\overline{\Theta}$, one has $T=+\infty$, and then
$$0\le u(t,x,\theta)\leq w(t,\theta)=\ep\,\textup{e}^{-\mu t}\varphi_\alpha(\theta)$$
for all $(t,x,\theta)\in[0,+\infty)\times\R\times\overline{\Theta}$. Therefore, $\|u(t,\cdot,\cdot)\|_{L^\infty(\R\times\Theta)}\to0$ as $t\to+\infty$. Lastly, since $\mu\in(0,\lambda_\alpha)$ and $\ep$ as in~\eqref{choiceeps} were arbitrary, the condition~\eqref{u0extinction} is therefore a sufficient condition for extinction, and this completes the proof.
\end{proof}

\subsection{Is it true that $\tmin<1/2$ makes persistence possible?}\label{sec:poss_persis} 

Finally, we focus on the last remaining question, that is the possibility of persistence when $\tmin<1/2$. It turns out to be a challenging problem,\footnote{Note that the case $\lambda_\alpha\leq 0$ (which includes the case $\tmin+\tmax\leq 0$) is already solved, since it implies systematic persistence. However we will not use this observation.} with several cases to be distinguished according to the sign of $\tmin+\tmax-1$. That corresponds to columns~3,~4,~5 and~6 of Table~\ref{table:summary}.

\subsubsection{The sub-critical case $\tmin+\tmax<1$}

\begin{thm}[Possible persistence in the sub-critical case]\label{thm:poss_persis_subcrit}
If $\tmin+\tmax<1$, then there exist initial data of the Cauchy problem \eqref{eq}--\eqref{initial} that lead to persistence.
\end{thm}

We begin with some preliminary lemmas.

\begin{lem}[On the preservation of monotonicity in $\theta$]\label{lem:decay} Assume that $M\le 1$ and that~$u_0$ is nonnegative, of class $\mathcal{C}^1_c(\R\times\overline{\Theta})$ and nonincreasing with respect to $\theta\in\overline{\Theta}$. Then, the solution $u(t,x,\theta)$ of the Cauchy problem~\eqref{eq}--\eqref{initial} is a nonincreasing function of $\theta$, for each $t\ge0$ and $x\in\R$.
\end{lem}

\begin{proof}
First of all, from standard parabolic estimates and bootstrap arguments, the function $u$ is of class $\mathcal{C}^\infty((0,+\infty)\times\R\times\overline{\Theta})$ and, from the regularity of $u_0$ and similar arguments as in Section~\ref{sec21}, it follows that $v:=u_\theta$ is continuous in $[0,+\infty)\times\R\times\overline{\Theta}$ and locally bounded with respect to $t$. Moreover, we have $v(t,x,\tmin)=v(t,x,\tmax)=0$ for all $t\ge0$ and $x\in\R$, and differentiating~\eqref{eq} with respect to $\theta$, we get 
$$v_t=d v_{xx}+\alpha v_{\theta \theta}+v(\rho-\theta)(1-\rho)-u(1-\rho)$$
in $(0,+\infty)\times\R\times\overline{\Theta}$. From \eqref{eq:max_rho} and the assumption $M\le1$, one has $\rho \le 1$ in $[0,+\infty)\times\R$, thus $-u \,(1-\rho)\le 0$ and
$$v_t\le d v_{xx}+\alpha v_{\theta \theta}+v(\rho-\theta)(1-\rho)$$
in $(0,+\infty)\times\R\times\overline{\Theta}$. But since $u_\theta=v$ is nonpositive at initial time by assumption, the parabolic maximum principle implies that $u_\theta(t,x,\theta)=v(t,x,\theta)\le0$ for all $t\ge0$ and~$(x,\theta)\in \R\times\overline{\Theta}$.
\end{proof}

\begin{lem}[An upper bound for the mean trait]\label{lem:thetabarsup}
Under the assumptions of Lemma~$\ref{lem:decay}$, we have
$$\tb(t,x)\le\frac{\tmin+\tmax}{2},$$
for all $t>0$ and $x\in \R$.
\end{lem}

\begin{proof}
Set
\begin{equation}\label{def:ubar}
\ub(t,x):=\frac{1}{\tmax-\tmin}\int_{\Theta} u(t,x,\theta) \, \textup{d}\theta=\frac{1}{\tmax-\tmin}\,\rho(t,x),\ \hbox{ for }t\ge0,\ x\in\R,
\end{equation}
and observe that the mean trait $\tb(t,x)$ satisfies, for $t>0$ and $x\in\R$,
\begin{equation} \label{eq:thetabar_decomp}
\begin{array}{rl}
\tb(t,x) &= \ds \frac{1}{\rho(t,x)}\int_{\Theta} \theta\, u(t,x,\theta) \, \textup{d}\theta, \vspace{2mm} \\
&=  \ds \frac{1}{\rho(t,x)}\int_{\Theta} \theta\, \ub (t,x) \, \textup{d}\theta + \frac{1}{\rho(t,x)}\int_{\Theta} \theta\, (u(t,x,\theta)-\ub(t,x)) \, \textup{d}\theta, \vspace{2mm}\\
& \ds =\frac{\tmin+\tmax}{2} +\frac{1}{\rho(t,x)}\int_{\Theta}\theta\, (u(t,x,\theta)-\ub(t,x)) \, \textup{d}\theta. 
\end{array}
\end{equation}
From Lemma~\ref{lem:decay} we know that, for each $(t,x)\in[0,+\infty)\times\R$, the function $\theta\mapsto u(t,x,\theta)-\ub(t,x)$ is nonincreasing with mean value $0$ over $\Theta$. Since $\theta\mapsto\theta$ is of course increasing in $\overline{\Theta}$, Chebyshev's integral inequality implies that 
$$(\tmax-\tmin)\times\int_{\Theta} \theta\, (u(t,x,\theta)-\ub(t,x)) \, \textup{d}\theta\le\Big(\int_{\Theta}\theta\, \textup{d}\theta\Big)\times\Big(\int_{\Theta}(u(t,x,\theta)-\ub(t,x))\, \textup{d}\theta\Big) = 0.$$
With \eqref{eq:thetabar_decomp}, this completes the proof of Lemma~\ref{lem:thetabarsup}.
\end{proof}

We are now in the position to complete the proof of Theorem \ref{thm:poss_persis_subcrit}.

\begin{proof}[Proof of Theorem~$\ref{thm:poss_persis_subcrit}$] Assume that $\tmin+\tmax<1$, and take $u_0$ as in Lemmas~\ref{lem:decay} and~\ref{lem:thetabarsup}. Using Lemma~\ref{lem:thetabarsup}, we have
$$\tb(t,x)\le\theta^*:=\frac{\tmin+\tmax}{2}<\frac12\ \hbox{ for all $t>0$ and $x\in \R$.}$$
Thus, we deduce from \eqref{eq-rho} and the comparison principle (recall that, here, $0\leq \rho \leq 1$ in~$[0,+\infty)\times\R$ since $M\le1$), that $\rho(t,x) \geq \underline \rho(t,x)$ for all $t\ge0$ and $x\in \R$, where $\underline \rho$ denotes the solution of the Cauchy problem
\begin{equation}\label{eq-theta-etoile}
\underline{\rho}_t=d\,\underline{\rho}_{xx}+ \underline{\rho} (\underline{\rho}-\theta^*)\,(1-\underline{\rho}),\ \ t>0,\ x\in\R,
\end{equation}
starting from $\underline{\rho}(0,\cdot)=\rho(0,\cdot)$ in $\R$. As $\theta^*<1/2$, standard results of~\cite{AW-78} imply the existence of initial conditions $\rho_0^*\in\mathcal{C}_c(\R,[0,1))$ such that the solution $\rho^{*}=\rho^{*}(t,x)$ of \eqref{eq-theta-etoile} starting from~$\rho_0^{*}$ satisfies $\rho^{*}(t,x) \to 1$ as $t \to +\infty$, locally uniformly in $x\in\R$. It is then sufficient to choose a nonnegative initial condition $u_0\in\mathcal{C}^1_c(\R\times\overline{\Theta})$, which is nonincreasing in $\theta$ and such that~$1\ge\rho(0,\cdot)\ge \rho_0^{*}$ in $\R$,\footnote{This inequality is satisfied for instance if $u_0=1/(\tmax-\tmin)$ in $[-R,R]\times\overline{\Theta}$, with $[-R,R]$ containing the support of $\rho_0^*$.} to get that $1\geq \rho(t,x)\ge \underline{\rho}(t,x)\geq \rho^{*}(t,x) \to 1$ as $t\to +\infty$ locally uniformly in $x\in\R$. Then $\rho(t,x)\to1$ as $t\to+\infty$ locally uniformly in~$x\in\R$ and such a solution $u$ then persists (and it even satisfies~\eqref{limits2}, as in the last part of Step~1 of the proof of Theorem~\ref{th:rescue}). The proof of Theorem~\ref{thm:poss_persis_subcrit} is thereby complete. 
\end{proof}






\subsubsection{The super-critical case $\tmin<1/2<(\tmin+\tmax)/2$}\label{sec36}

This case corresponds to column 6 of Table~\ref{table:summary}. Since $\tmax<1$, one then has $\tmin>0$ in this case, hence $\lambda_\alpha>0$ by Lemma~\ref{lem:vp}. Therefore, by Theorem~\ref{th:ext2}, sufficiently small initial data $u_0$ of the Cauchy problem~\eqref{eq}--\eqref{initial} lead to extinction. Although the assumption $\tmin\ge1/2$ leads to systematic extinction by Theorem~\ref{thm:syst-ext}, the case $\tmin<1/2<(\tmin+\tmax)/2$ is more subtle and is handled with different techniques, leading to the identification of new parameter regimes.

\begin{thm}[Systematic extinction in the super-critical case with large $\alpha$]\label{thm:poss_persis_supercrit}
If $\tmin<1/2$ and $\tmin+\tmax>1$, then all solutions of the Cauchy problem~\eqref{eq}--\eqref{initial} go extinct, provided $\alpha>\alpha^\star$, for some $\alpha^\star>0$. Moreover, $\alpha^\star\le \alpha^\sharp,$ where $\alpha^\sharp$ has an explicit form which depends only on $M$, $\tmin$ and $\tmax$.
\end{thm}

\begin{proof} 
As already underlined, one here has $\tmin>0$. Consider for the moment any~$\alpha>0$ and any solution $u$ of~\eqref{eq}--\eqref{initial}. Define
$$v=u_\theta$$
and, for $t>0$ and $x\in\R$,
$$V(t,x):=\frac{1}{2}\|v(t,x,\cdot)\|_{L^2(\Theta)}^2.$$
From standard parabolic estimates and the global boundedness of $u$ and $\rho$, the function~$u$ is of class $\mathcal{C}^\infty((0,+\infty)\times\R\times\overline{\Theta})$, and $u_\theta$ is bounded in $[\ep,+\infty)\times\R\times\overline{\Theta}$ for each $\ep>0$. Hence, the function $V$ is of class $C^\infty((0,+\infty)\times\R)$ and bounded in $[\ep,+\infty)\times\R$ for each~$\ep>0$.

We derive in this paragraph a partial differential inequality satisfied by $V(t,x)$ for any fixed $(t,x)\in(0,+\infty)\times\R$. All quantities below involving $\rho$, $V$ and the partial derivatives of $V$, are evaluated at $(t,x)$. Differentiating~\eqref{eq} with respect to $\theta$, multiplying by $v$ and integrating over $\Theta$, we get that
$$V_t=d\int_\Theta v\, v_{xx} \, \textup{d}\theta+\alpha \int_\Theta v\, v_{\theta\theta} \, \textup{d}\theta+(1-\rho)\int_\Theta v^2(\rho-\theta) \, \textup{d}\theta- (1-\rho)\int_\Theta u\, v \, \textup{d}\theta.$$
Integrating by parts (with $v\equiv0$ on $(0,+\infty)\times\R\times\partial\Theta$) and using $0\leq \rho(t,x)\leq\max(M,1)$ and $\rho(t,x)(1-\rho(t,x))\leq1/4$, we obtain
$$V_t \le d\int_\Theta v\,v_{xx}\,\textup{d}\theta- \alpha \int_\Theta (v_{\theta})^2\,\textup{d}\theta +\left(\frac12+2(M+1)\tmax\right) V +(M+1) 
\int_\Theta u\,|v|\,\textup{d}\theta.$$
Since
$$\int_\Theta v\,v_{xx}\,\textup{d}\theta=\int_\Theta \Big(\frac{v^2}{2}\Big)_{xx}\,\textup{d}\theta-\int_\Theta v_x^2\,\textup{d}\theta\leq V_{xx},$$
the Cauchy--Schwarz inequality leads to
$$V_t \le d\,V_{xx} - \alpha \|v_\theta\|_{L^2(\Theta)}^2 +\left(\frac12+2(M+1)\tmax\right)V+(M+1)\sqrt{2U}\sqrt{2V},$$
where
$$U(t,x):=\frac{1}{2}\|u(t,x,\cdot)\|_{L^2(\Theta)}^2.$$
Additionally, as $v=0$ on $(0,+\infty)\times\R\times\partial\Theta$, the Poincar\'e inequality yields
$$\|v_\theta(t,x,\cdot)\|_{L^2(\Theta)}^2\ge \lambda ^{\textup{D}}_1 \|v(t,x,\cdot)\|_{L^2(\Theta)}^2=2 \, \lambda ^{\textup{D}}_1 V(t,x),$$
with $\lambda ^{\textup{D}}_1>0$ the principal eigenvalue of $-\partial_{\theta \theta}$ in $\Theta$ with Dirichlet boundary conditions, that is,
\begin{equation}\label{lambda1D}
\lambda^{\textup{D}}_1:=\min \left\{\frac{\int _\Theta \varphi _\theta ^{2}}{\int _\Theta \varphi^{2}}: \varphi \in H^{1}_0(\Theta)\setminus\{0\} \right\}=\frac{\pi^2}{(\tmax-\tmin)^2}.
\end{equation}
Therefore,
\begin{equation}\label{a-suivre}
    V_t - d V_{xx} \le \sqrt{V}\lp \Big(\frac12+2(M+1)\tmax-2 \alpha\lambda ^{\textup{D}}_1\Big)\sqrt{V} + 2(M+1)\sqrt{U}\rp.
\end{equation}
Next, after recalling the definition~\eqref{def:ubar} of $\overline{u}(t,x)$, the Poincar\'{e}--Wirtinger inequality implies that
\begin{equation}
\label{wirtinger}
2V(t,x)=\|u_\theta(t,x,\cdot)\|_{L^2(\Theta)}^2\ge\lambda^{\textup{N}}_2\|u(t,x,\cdot)-\ub(t,x)\|_{L^2(\Theta)}^2,
\end{equation}
with $\lambda ^{\textup{N}}_2>0$ the smallest nonzero eigenvalue of $-\partial_{\theta \theta}$ in $\Theta$ with Neumann boundary conditions, that is 
\begin{equation}\label{lambda2N}
\lambda^{\textup{N}}_2:=\min \left\{\frac{\int _\Theta \varphi _\theta ^{2}}{\int _\Theta \varphi^{2}}: \varphi \in H^{1}(\Theta)\setminus\{0\}, 
\int _\Theta \varphi =0 \right\}=\frac{\pi^2}{(\tmax-\tmin)^2}=\lambda^{\textup{D}}_1.
\end{equation}
In particular, we deduce from $\sqrt{2U(t,x)}\leq \|u(t,x,\cdot)-\ub(t,x)\|_{L^2(\Theta)}+\| \ub(t,x) \| _{L^2(\Theta)}$ that
\begin{equation}\label{to-be-plugged}
2(M+1)\sqrt{U(t,x)}\leq \frac{2(M+1)}{\sqrt{\lambda _2^{\textup{N}}}} \sqrt{V(t,x)}+\frac{\sqrt{2}(M+1)^2}{\sqrt{\tmax-\tmin}},
\end{equation}
since $0\leq \ub (t,x)\leq\max(M,1)/(\tmax-\tmin)\le(M+1)/(\tmax-\tmin)$ in view of~\eqref{eq:max_rho} and~\eqref{def:ubar}. Plugging \eqref{to-be-plugged} into \eqref{a-suivre}, we end up with
$$V_t-d\,V_{xx} \le \sqrt{V}\lp R - \mu_\alpha\, \sqrt{V} \rp,$$
where
\begin{equation}\label{defmualpha}
\mu_\alpha=2\alpha\lambda_1^{\textup{D}}-\frac12-2(M+1)\tmax-\frac{2(M+1)}{\sqrt{\lambda_2^{\textup{N}}}}\ \hbox{ and }\  R=\frac{\sqrt{2}(M+1)^2}{\sqrt{\tmax-\tmin}}.
\end{equation}

From~\eqref{lambda1D},~\eqref{lambda2N} and~\eqref{defmualpha}, there exists $\alpha^\sharp_1>0$, depending only on~$M$,~$\tmin$ and~$\tmax$, such that
$$\mu_\alpha>0\ \hbox{ for all }\alpha>\alpha^\sharp_1.$$ A straightforward computation shows that $$\alpha^\sharp_1=\frac{(\tmax-\tmin)^2}{\pi^3}\lp (M+1)(1+\pi)\tmax-(M+1)\tmin+\frac{\pi}{4} \rp.$$
From now on, we assume that $\alpha>\alpha_1^\sharp$. By comparison with the explicit solution $\overline{V}$ of the underlying ordinary differential equation starting at time $t=1$ from $\sup_{\mathbb{R}}V(1,\cdot)$ (which is a nonnegative real number), we get that
\begin{equation}
\label{V-bar}
V(t,x)\le \lp \frac{R}{ \mu_\alpha}\!+\! \Big(\sup_{\mathbb{R}}V(1,\cdot)-\frac R{\mu_\alpha}\Big) \, \textup{e}^{-\mu _\alpha(t-1)/2}  \rp^2:=\overline{V}(t)\ \text{ for all }t\geq 1,\ x\in\R.
\end{equation}
Next, coming back to \eqref{eq:thetabar_decomp}, namely
$$\tb(t,x) = \frac{\tmin+\tmax}{2} +\frac{1}{\rho(t,x)}\int_{\Theta} \theta\, (u(t,x,\theta)-\ub(t,x)) \, \textup{d}\theta,$$
we obtain, from the Cauchy--Schwarz inequality together with~\eqref{wirtinger} and~\eqref{V-bar},
\begin{equation} \label{minoration}
         \tb(t,x) \ge \frac{\tmin+\tmax}{2} -\frac{K}{\rho(t,x)} \sqrt{\overline{V}(t)}
         \quad\text{for all }t\geq1\hbox{ and }x\in\R, 
\end{equation}
where $K=\sqrt{2(\tmax^3-\tmin^3)/(3\lambda_2^{\textup{N}})}>0$ is a constant that only depends on~$\tmin$ and~$\tmax$.

Now, remembering that $\tmin+\tmax>1$ and $\tmin>0$, let $\eta>0$ (only depending on~$\tmin$ and~$\tmax$) be small enough so that
\begin{equation}\label{defeta}
\tmin+\tmax>1+ 10 \eta\ \hbox{ and }\ 0<\eta<\frac{\tmin}{4}.
\end{equation}
From~\eqref{defmualpha}, it follows that there exists $\alpha^\sharp\ge\alpha^\sharp_1$ only depending on $M$, $\tmin$ and $\tmax$, such that
\begin{equation}
\label{inegalite}
\frac{R}{\mu_\alpha}<\frac{\eta^2}{K}\ \hbox{ for all }\alpha>\alpha^\sharp.
\end{equation}
This threshold can be computed explicitly:
\begin{multline}\label{eq:alpha_sharp}
\alpha^\sharp=\frac{(\tmax-\tmin)^2}{\pi^3}\bigg( (M+1)(1+\pi)\tmax-(M+1)\tmin+\frac{\pi}{4} \\ +\frac{(M+1)^2}{(\eta^*)^2\, \sqrt{3}}\sqrt{(\tmax^3-\tmin^3)(\tmax-\tmin)}\bigg),   \end{multline}
with $\eta^*=\min\big((\tmax+\tmin-1)/10,\tmin/4\big)$.
From now on, we assume that
$$\alpha>\alpha^\sharp.$$
From \eqref{V-bar} and \eqref{inegalite}, it follows that there exists a time $T_1\ge1$ such that $K\sqrt{\overline{V}(t)}<\eta^2$ for all $t\ge T_1$ (notice that~$T_1$ depends on $\alpha$ and also of $\sup_\R V(1,\cdot)$ and then also on $u$ itself, but this does not matter since we are only concerned with the extinction, at long time, of $u$). Observe also that, defining $\theta^*:=1/2 + 2\eta$, \eqref{minoration} insures that
\begin{equation} \label{eq:thetabar_eps}
\text{ if }\;   t\geq T_1\text{ and }\rho(t,x)\geq \eta,\; \text{ then }\;     \tb(t,x) \ge \frac{\tmin+\tmax}{2} -\eta>\frac{1}{2}+4\eta>\theta^*.
\end{equation}

Define 
\begin{equation*}
\overline{f}(s)=(s-2\eta)\left(s-\theta^*\right)(1-s)+\eta^2 (s-2\eta)(s-\theta^*)\mathbf{1}_{[\theta^*,1+\eta^2]}(s),\ \hbox{ for }s\ge0.
\end{equation*}
Since $1+10\eta<\tmin+\tmax<2$, one has $\eta<1/10$, so that
$$2\eta<\frac12<\theta^*=\frac12+2\eta<1<1+\eta^2$$
and $\overline{f}$ is a bistable reaction term with stable steady states $2\eta$ and~$1+\eta^2$ and unstable steady state $\theta^*=1/2+2\eta$. By straightforward computations, one has
\begin{equation*}
\int _{2\eta}^{1+\eta^2}\overline{f}(s)\textup{d}s=-\frac{\eta}{6}+o(\eta)\quad\text{as }\eta\to 0.
\end{equation*}
Hence we may assume without loss of generality, up to reducing $\eta>0$ (depending on~$\tmin$ and~$\tmax$ only), that $\int _{2\eta}^{1+\eta^2}\overline{f}<0$. Define then
\begin{equation*}
T_2 = \inf\left\{ t\geq T_1:\forall\,\tau\geq t,\ \sup_\R\rho(\tau,\cdot)\leq 1+\frac{\eta^2}{2}\right\}.
\end{equation*}
The time $T_2$, which depends on $u$ and the other parameters of the problem, is well-defined and finite by virtue of~\eqref{eq:limsup-mass}, that is, $1\le T_1\le T_2<+\infty$. Consider now the solution $\rhob=\rhob(t,x)$ of the bistable reaction--diffusion equation
\begin{equation}\label{eq-rhob}
\rhob_t=d\rhob _{xx}+\overline{f}(\rhob), \quad t>T_2,\ x\in \R,
\end{equation}
starting from $\rhob(T_2,x)=\max(\rho(T_2,x),2\eta)$. Notice that $2\eta\le\rhob(T_2,x)\le1+\eta^2/2<1+\eta^2$ for all $x\in\R$, and that $\rhob(T_2,x)\to2\eta$ as $x\to\pm\infty$ by~\eqref{zero-infini}. Therefore, it follows from \cite{Fife_McLeod_19} (as in the proof of Theorem~\ref{thm:syst-ext} in the case $\tmin>1/2$, see Section~\ref{sec:imposs_persis}) that
\begin{equation}
\label{rhobt}
\rhob(t,x)\to 2\eta\; \text{ as $t\to +\infty$, uniformly in $x\in\R$.}
\end{equation}

We finally claim that
$$f(t,x,\rho(t,x)):=\rho(t,x)\,(\rho(t,x)-\tb(t,x))\,(1-\rho(t,x))\leq \overline{f}(\rho(t,x))\ \text{ for all } t\geq T_2,\ x\in \R.$$
Indeed, for any $t\ge T_2$ and $x\in\R$, one has on the one hand $0<\tmin<\tb(t,x)<\tmax<1$ and $0<\rho(t,x)\le1+\eta^2/2<1+\eta^2$, and on the other hand:
\begin{itemize}
\item when $1\leq \rho(t,x)\leq 1+\eta^2/2$, $f(t,x,\rho(t,x))-\overline{f}(\rho(t,x))$ is obviously nonpositive since $f(t,x,\rho(t,x))\le0$ and $\overline{f}(\rho(t,x))\ge0$;
\item when $\eta\le\rho(t,x)<1$, then
$$f(t,x,\rho(t,x))-\overline{f}(\rho(t,x))\leq f(t,x,\rho(t,x))-(\rho(t,x)-2\eta)\,(\rho(t,x)-\theta^*)\,(1-\rho(t,x))$$
and the sign of the right-hand side is that of $\rho(t,x)(\theta^*-\overline \theta(t,x)+2\eta)-2\eta \theta ^*$, which is negative in view of~\eqref{eq:thetabar_eps};
\item when $0<\rho(t,x)<\eta$, the sign of 
$$f(t,x,\rho(t,x))-\overline{f}(\rho(t,x))=(1-\rho(t,x))\,\big[\theta^*(\rho(t,x)-2\eta)+\rho(t,x)(2\eta-\tb(t,x))\big]$$
is that of $\theta^*(\rho(t,x)-2\eta)+\rho(t,x)(2\eta-\tb(t,x))$; but, since $\tb(t,x)>\tmin>0$, there holds $\theta^*(\rho(t,x)-2\eta)+\rho(t,x)(2\eta-\tb(t,x))<-\eta\theta^*+2\eta^2=-\eta/2<0$.
\end{itemize}
As a result, recalling~\eqref{eq-rho}, $\rho$ is then a subsolution of the equation~\eqref{eq-rhob} satisfied by $\rhob$ for times~$t\ge T_2$, with $0<\rho(T_2,\cdot)\le\rhob(T_2,\cdot)$ in $\R$, and therefore, $0<\rho (t,x)\le  \rhob(t,x)$ for all~$t\ge T_2$ and~$x\in\R$ from the maximum principle. 

Since $\eta<\tmin/4$ by~\eqref{defeta}, we deduce from~\eqref{rhobt} the existence of a time $T_3\geq T_2$ such that
$$0<\rho (t,x)<\frac{\tmin}{2}<\frac12<1\ \hbox{ for all } t\geq T_3,\ x\in \R.$$
Since $\tb (t,x)>\tmin$ in $(0,+\infty)\times\R$ and $\tmin<\tmax<1$, it then follows from~\eqref{eq-rho} and the previous inequality that 
$$\rho _t \leq d \rho _{xx}+\rho(\rho -\tmin)(1-\rho)\quad\text{for all }t\geq T_3,\ x\in\R.$$
Hence, by comparison with the underlying bistable ordinary differential equation, one infers that $\rho (t,x)\to 0$ as $t\to +\infty$, uniformly in $x\in \R$. In other words, $u$ goes extinct, as soon as $\alpha>\alpha^\sharp$, with $\alpha^\sharp>0$ given by \eqref{eq:alpha_sharp} and thus only depending on $M$, $\tmin$ and~$\tmax$. The proof of Theorem~\ref{thm:poss_persis_supercrit} is thereby complete.
\end{proof}

\section{Numerical results \label{sec:numerics}}

The objective of this section is to get an overview of the shape of the solution $u(t,x,\theta)$ of~\eqref{eq} and to test the conjectures made in Section~\ref{sec:conj}. We solved the equation~\eqref{eq} on a rectangular domain $(x,\theta)\in I\times (\tmin,\tmax)$ with a ``method of lines", using the Matlab$^\circledR$ ode45 solver (the source code is available in the Open Science Framework repo\-sitory: \url{https://osf.io/w8nuz/}). We considered characteristic functions of sets of various dimensions $L\times (\tmin,\tmax)$ as initial conditions (though they are not continuous, these functions can be approximated by continuous functions without changing much the numerics):
\begin{equation} \label{eq:DI_num}
u_0(x,\theta)=\frac{1}{\tmax-\tmin}\mathds{1}_{(-L/2,L/2)}(x)\; \hbox{ for }(x,\theta)\in \underbrace{(-60,60)}_{=:I}\times(\tmin,\tmax),
\end{equation}
with $0<L\le 80$.

\begin{figure}[ht]
\center
\subfigure[]{\includegraphics[width=0.8\textwidth]{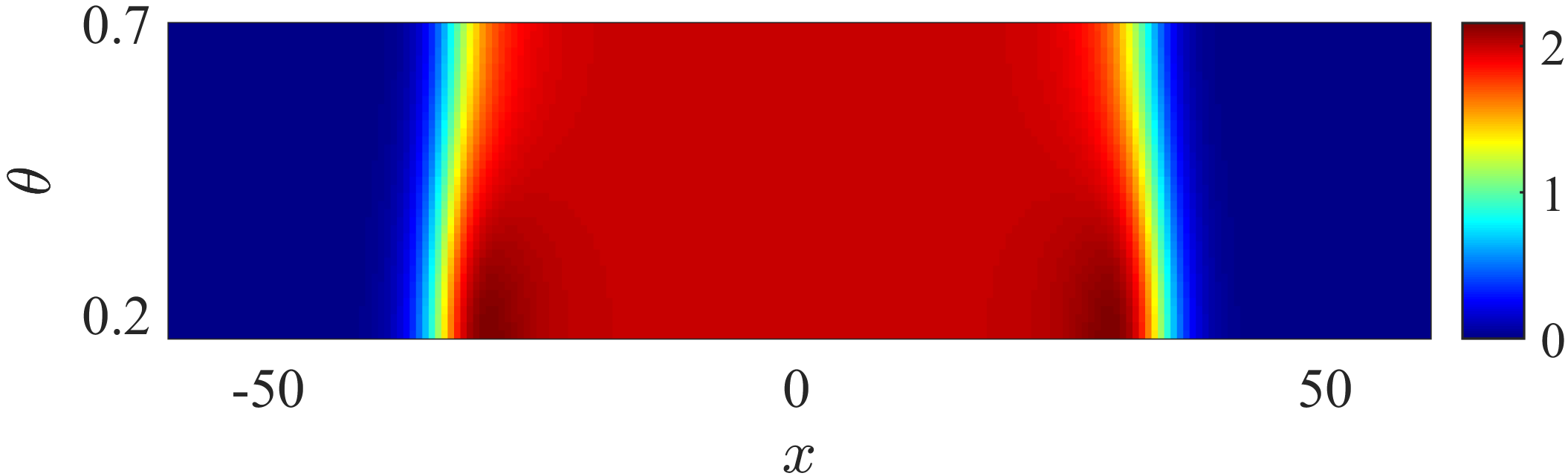}}
\subfigure[]{\includegraphics[width=0.82\textwidth]{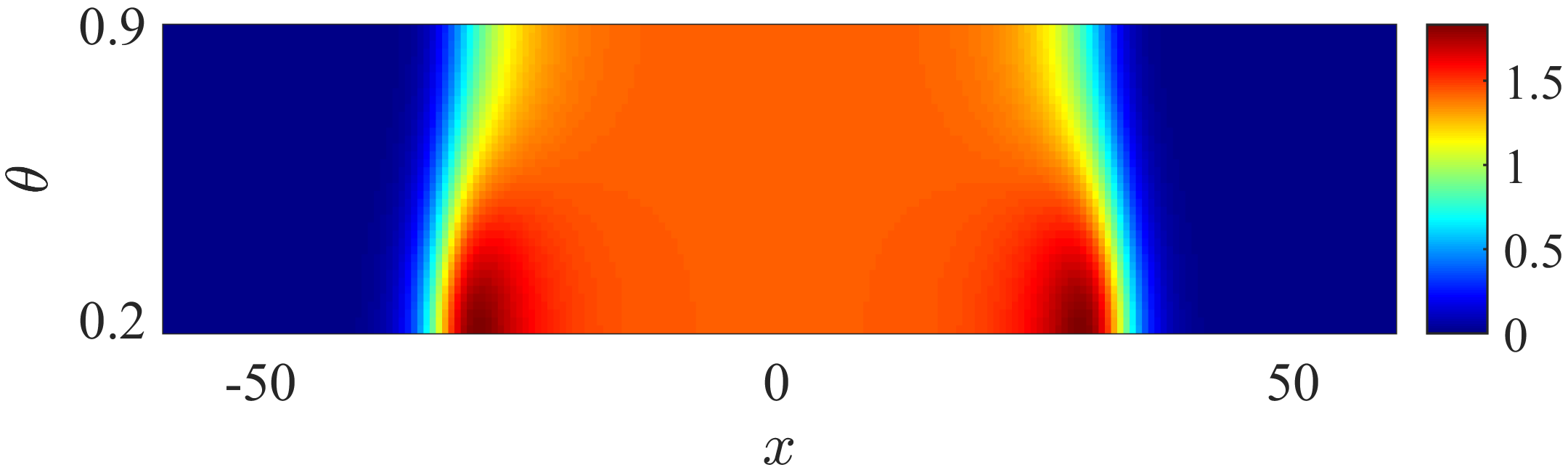}}
\caption{{\bf Numerical solution $u(t,x,\theta)$ of \eqref{eq} at some fixed time}. In panel~(a), we have taken $\tmin=0.2$ and $\tmax=0.7$, and the solution is computed at $t=200.$ In panel~(b), $\tmin=0.2$ and $\tmax=0.9$ and the solution is computed at $t=400$. In both cases, the initial condition is given by \eqref{eq:DI_num} with $L=20$. The other parameter values are: $d=1$ and $\alpha=4\cdot 10^{-3}$.}
\label{fig:distri}
\end{figure}

We depict the shape of $u(t,x,\theta)$ at some fixed positive time in Fig.~\ref{fig:distri} in the subcri\-tical case ($\tmin+\tmax<1$, panel~(a) and supercritical case ($\tmin+\tmax>1$, panel~(b). Interestingly,  the solution takes its highest values when $x$ is close to the leading edge of the front, i.e., at the transition zones between $\rho \approx 0$ and $\rho \approx 1$. At such positions, the population tends to concentrate on trait values close to $\tmin$. This is consistent with the interpretation of a stronger selection pressure due to the Allee effect at low density (see the biological motivation part of the Introduction). Conversely, in the ``core" of the population, the solution tends to get flatter, which reflects the convergence of the mass $\rho$ towards the value $1$, which in turns implies that the reaction term in \eqref{eq} is close to $0$. Thus, in this central region, the dynamics is mainly driven by diffusion (spatial diffusion and mutations). 



We now test our conjectures. To construct Fig.~\ref{fig:persistence}, we solved the equation~\eqref{eq} until a time $T=10^3$, for increasing values of the mutation parameter $\alpha$ (with step $10^{-4}$) and of the length $L$ of the support of $u_0$ (with step $1$). Each time, we computed the total mass $N(t)=\int_I \rho(t,x) \, \textup{d} x$ for $t\in [0,T]$. We considered that persistence occurred if $N(T)>|I|-1=119$ ($\rho \approx 1$ over the whole domain $I\times (\tmin,\tmax)$); that persistence was probable if $N(T)>N(T/2)$; that extinction was probable if  $N(T)<N(T/2)$; and that extinction occurred if $N(T)<1$. In the critical case $\tmin+\tmax=1$, we conjectured in Section~\ref{sec:conj} that, for any value of the mutation parameter $\alpha$, extinction or persistence can both occur according to the initial condition. This is fully consistent with the numerical results in Fig.~\ref{fig:persistence}a. In the supercritical case $\tmin+\tmax>1$, we proved that for $\alpha$ large enough, extinction was systematic. This corresponds to the region $\alpha > \alpha^\star \approx 7.5 \cdot 10^{-3}$ in Fig.~\ref{fig:persistence}b. In this plot, we also observe that, as conjectured, when $\alpha$ is below this threshold, extinction or persistence can both occur depending on $u_0$ (here $L$) if $\alpha \le \alpha^\star$. Note that, with the parameter values in Fig.~\ref{fig:persistence}b, the formula \eqref{eq:alpha_sharp} leads to $\alpha^\sharp\approx 259$ which is far from optimal. 

\begin{rem} Close to the critical threshold $\alpha^\star\approx 7.5 \cdot 10^{-3}$  and for  $L\approx 15$ we observe a small pink \lq\lq persistence region'' encroached below the cyan \lq\lq extinction region''. At first glance, this may appear surprising since larger values of $L$ are expected to lead to higher chances of persistence (even though the comparison principle does not hold). A closer look at the solution of \eqref{eq} (not depicted here) for $(\alpha,L)$ in this region shows that the solution seems to converge to a stationary state, either by increasing its total mass $N(t)$ (for small $L$, pink region) or by decreasing it (for large $L$, cyan region), which explains the pattern in Fig.~\ref{fig:persistence}b.
\end{rem}

\begin{figure}[ht]
\center
\subfigure[]{\includegraphics[width=0.48\textwidth]{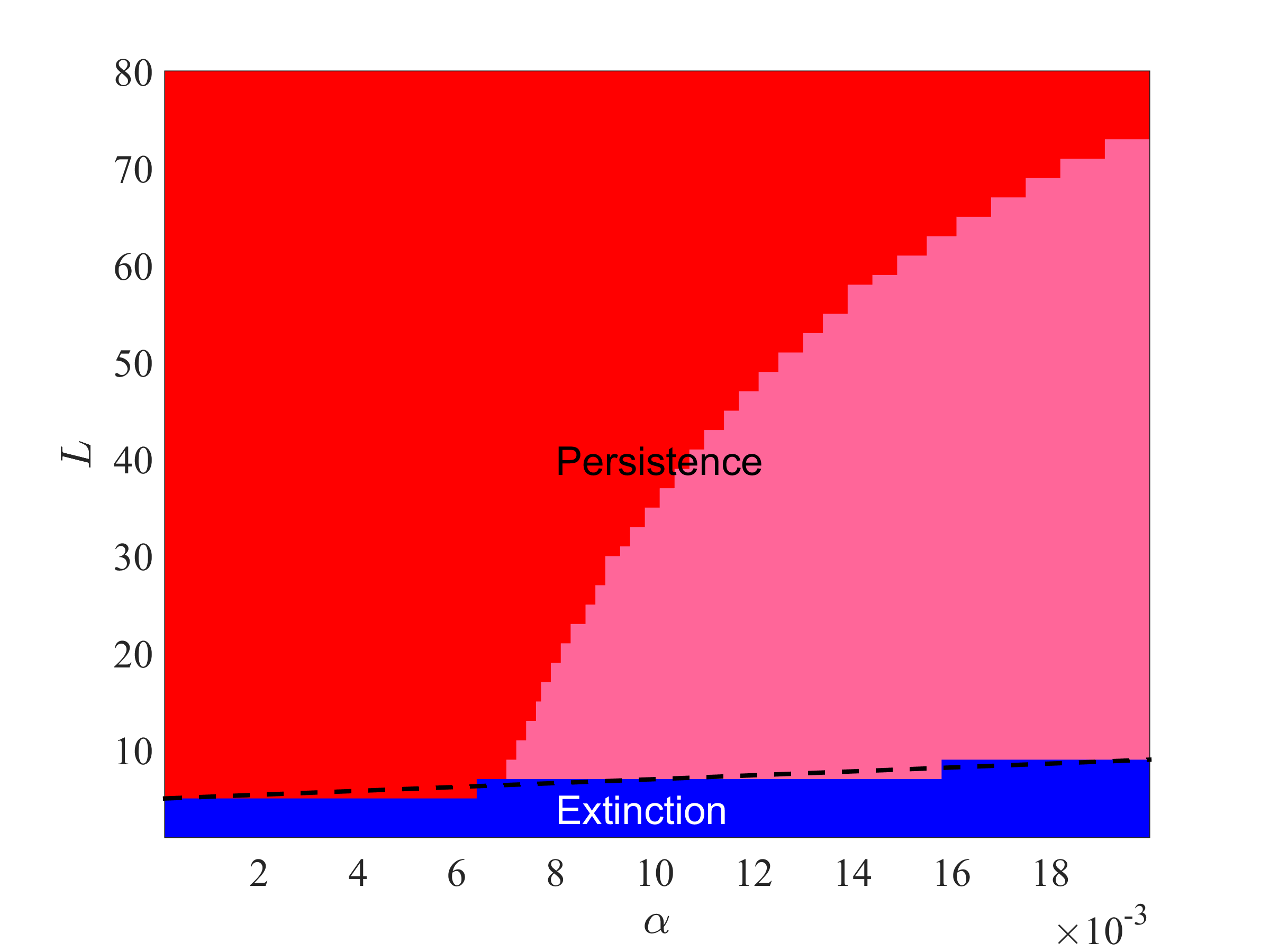}}
\subfigure[]{\includegraphics[width=0.48\textwidth]{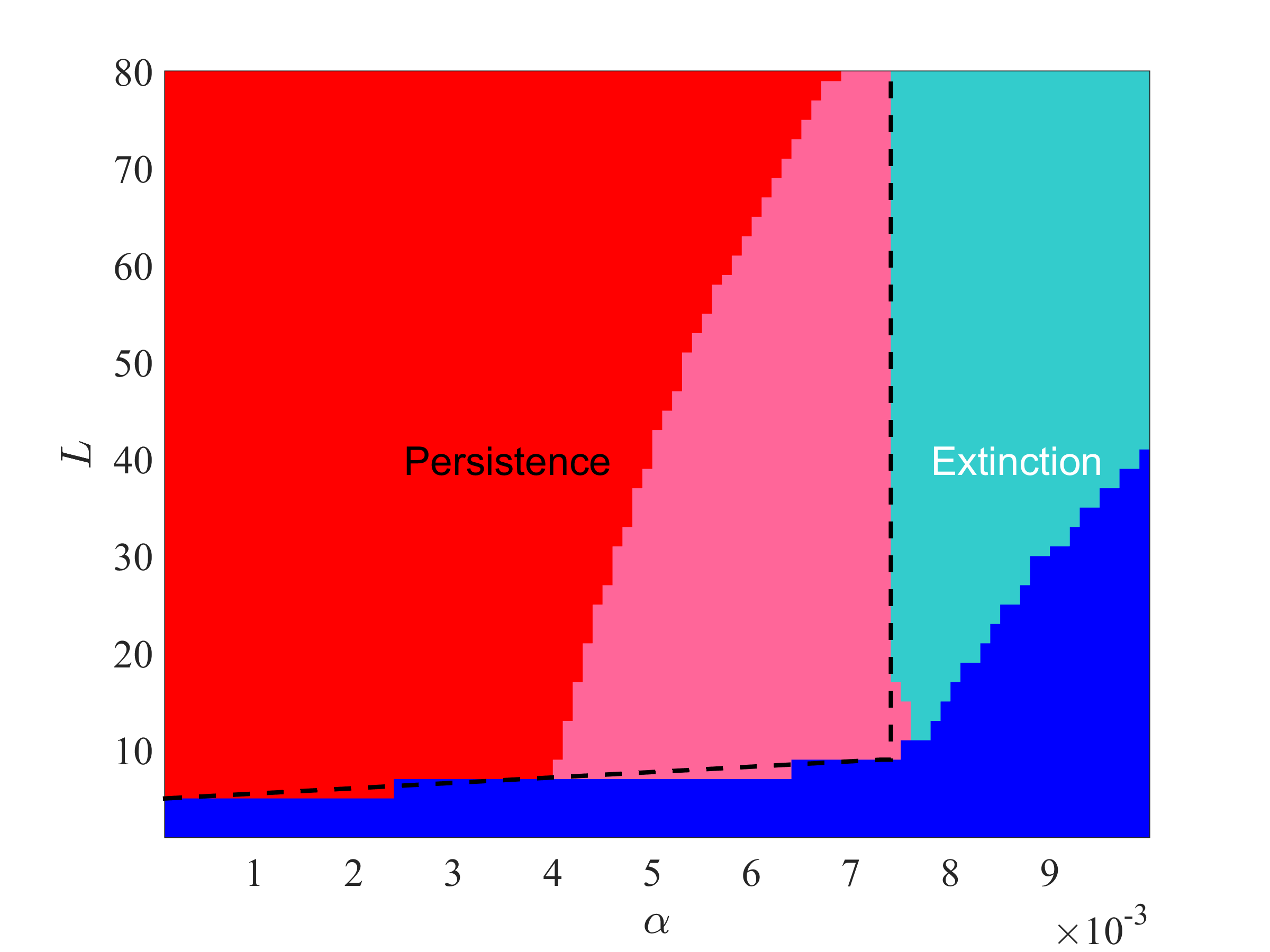}}
\caption{{\bf Persistence at large times, in terms of the mutation parameter $\alpha$ and of the size $L$ of the support of $u_0$}. Panel~(a) corresponds to $\tmin=0.2$ and $\tmax=0.8$, so that $\tmin+\tmax=1$.  Panel~(b) corresponds to $\tmin=0.2$ and $\tmax=0.9$, so that $\tmin+\tmax=1.1>1$. In both cases, the initial condition is given by \eqref{eq:DI_num} and the solution $u(t,x,\theta)$ of \eqref{eq} is evaluated at $T=10^3$. The red region corresponds to persistence; the pink region to probable persistence; the cyan region to probable extinction; the blue region corresponds to extinction. The spatial diffusion coefficient was fixed at $d=1$.}
\label{fig:persistence}
\end{figure}

We also checked the dependence of the persistence/extinction behavior with respect to the value of the dominant trait $\tilde{\theta}$ in the initial population. Namely, we considered initial conditions that are non-uniform with respect to $\theta$:
\begin{equation} \label{eq:DI_num_gauss}
u_0(x,\theta)=C_{\tilde\theta} \, \exp\lp-\frac{(\theta-\tilde{\theta})^2}{2 \sigma^2}\rp \mathds{1}_{(-L/2,L/2)}(x)\; \hbox{ for }(x,\theta)\in I\times(\tmin,\tmax),
\end{equation}
with again $I=(-60,60)$,  $\tilde{\theta}\in \Theta$ and $C_{\tilde\theta}$ such that $\int_{\Theta}u_0(x,\theta)\,\textup{d}\theta=1$ for all $x\in (-L/2,L/2)$ as in the previous example \eqref{eq:DI_num}. We worked here with a single value of the mutation parameter $\alpha=5 \cdot 10^{-3}$ such that both extinction and persistence were observed with initial conditions of the form  \eqref{eq:DI_num}, depending on $L$. We chose here $L=5$ and $\sigma=0.1$. The other parameters were the same as in Fig.~\ref{fig:persistence}b, in particular $\tmin=0.2$ and $\tmax=0.9$. This time, with the same criteria as above, we observed that persistence occurred for $\tilde{\theta}$ below some critical threshold $\tilde{\theta}^\star \approx 0.5$ and extinction occurred for $\tilde{\theta}$ above this threshold. Note that, with $L=5$ and the same other parameter values, extinction occurred with a uniform initial distribution (Fig.~\ref{fig:persistence}b). Thus, keeping the same initial population size, the outcomes depends on the initial distribution of the trait, and persistence becomes more likely when the initial distribution is concentrated around $\tmin$. Video files illustrating the dynamics of $u$ and $\rho$ with initial conditions \eqref{eq:DI_num_gauss} for several values of $\tilde{\theta}$ are available in the Open Science Framework repo\-sitory: \url{https://osf.io/w8nuz/}, together with the corresponding Matlab$^\circledR$  source code.

\section*{Acknowledgements}
The authors are grateful to the reviewers for their insightful, detailed, and much helpful comments.

\bibliographystyle{siam}  

\bibliography{biblio,biblio_lionel_jab_drive}

\end{document}